\documentclass[11pt,reqno]{amsart}
\usepackage{amsmath, amsfonts, amsthm, amssymb, stmaryrd, color, cite}
\usepackage{parskip}
\usepackage{latexsym,hyperref}

\usepackage[dvipsnames]{xcolor}

\numberwithin{equation}{section}
\newtheorem{theorem}{Theorem}
\newtheorem{definition}{Definition}
\newtheorem{lemma}{Lemma}

\newtheorem{proposition}{Proposition}

\numberwithin{equation}{section}

\title[\emph{LIL for Navier-Stokes Equation}]
{The Law of the Iterated Logarithm for Two-Dimensional Stochastic Navier-Stokes Equations}
\date{}

\author[P. Fatheddin]{Parisa Fatheddin}
\address{Department of Mathematics, Ohio State University, Marion, USA.}
\email{fatheddin.1@osu.edu}


\begin{document}
\maketitle
\begin{abstract}
We implement the Azencott method to prove the moderate deviation principle for the two-dimensional incompressible stochastic Navier-Stokes equations in a bounded domain. The Strassen's compact law of iterated logarithm is then achieved as an application.
\end{abstract}
\section{\textbf{I\lowercase{ntroduction}}}
\hspace{.5cm} The Classical Azencott method, first introduced in [3, 31], is applied here to prove the moderate deviation principle (MDP) for the two dimensional stochastic Navier-Stokes equations in incompressible flow given by,
\begin{eqnarray*}
&&\frac{\partial u^{\varepsilon}(t)}{\partial t} + \left(u^{\varepsilon}(t)\cdot \nabla\right) u^{\varepsilon}(t) + \nabla p(t) = f(t) + \Delta u^{\varepsilon}(t) + \sqrt{\varepsilon} \sigma(t,u^{\varepsilon}(t)) \frac{dW(t)}{dt}, \\
&& \left(\nabla \cdot u^{\varepsilon}\right)(t,y)=0, \hspace{.5cm} y\in D, t>0,\\
&& u^{\varepsilon}(t,y)= 0, \hspace{.5cm} y\in \partial D, t\geq 0,\\
&& u^{\varepsilon}(0,y) = u_{0}(y), \hspace{.5cm} y\in D,
\end{eqnarray*}
where $p(t,y)$ and $u^{\varepsilon}(t,y)\in \mathbb{R}^{2}$ are the pressure and velocity fields, respectively. A deterministic external force $f(t,y)\in \mathbb{R}^{2}$ is assumed to be given, along with a noise coefficient $\sigma(t,u^{\varepsilon}(t))$ of a Wiener process $W(\cdot)$, with properties provided later in Section 2. \\
\hspace{.5cm} The Theory of large deviations has proved to be useful in many fields such as in statistical mechanics, finance, queuing theory and communications. For examples of applications we refer the reader to [9, 21, 23, 25, 27, 30]. Another closely related area of study is moderate deviations, for which one proves large deviations for the centered process multiplied by a rate of convergence slower than the rate used for large deviations. \\
\hspace{.5cm} The majority of work on large deviations regarding the stochastic Navier-Stokes equations has been established based on the weak convergence approach introduced by [8, 10]. In [16] this technique was applied to obtain the large deviation principle (LDP) for a general class of stochastic PDEs of which two-dimensional stochastic incompressible Navier-Stokes equation (SNSE) was a special case. For two-dimensional viscous, incompressible SNSE, this approach was taken in [35] where the noise term tends to zero and in [5] where the viscosity is set to vanish. As for moderate deviations, considering the same general class of SPDEs introduced in [16] to achieve large deviations, authors in [40] proved the MDP for this class with multiplicative noise and in unbounded domains. For MDP in a bounded domain, the weak convergence method was applied for two-dimensional viscous, incompressible SNSE with multiplicative noise in [38] and with L$\acute{e}$vy noise in [19]. Our aim here is to revisit a classical method for large deviations, called the Azencott method in order to bring another perspective and emphasize the applications it offers. By proving the MDP by this method, we obtain the Friedlin-Wentzell inequality, which plays a major role in proving the Strassen's compact LIL. \\
\hspace{.5cm} There are different forms of LIL in the literature with names Classical Khintchine LIL, Strassen's compact LIL, Chover's type LIL, and Chung's type LIL. For a detailed introduction and history on each type we recommend [6] and for applications in fields such as finance, we refer the reader to [24, 41]. Here we consider the Strassen's compact LIL for our model. After the observation first made by [18] (see Lemma 1.4.3), many authors have proved the Strassen's compact LIL as a consequence of their large or moderate deviation results (see for instance [1, 4, 14, 20, 29]). Similarly, we use our result on MDP to achieve the Strassen's compact LIL.\\
\hspace{.5cm} We begin in Section Two with some background on the law of the iterated logarithm and provide statements of the main results along with notations needed for the rest of the paper. Section Three is devoted to proving the MDP by the Azencott method and afterwards the compact LIL is established in Section Four.
\section{\textbf{P\lowercase{reliminaries and} M\lowercase{ain} R\lowercase{esults}}}
\hspace{.5cm} In this section, we provide the notations and background needed for the paper. Let $D\subset \mathbb{R}^{2}$ be a bounded open domain with smooth boundary $\partial D$. For convenience, we will denote $\textbf{u}^{\varepsilon}(t)\in \mathbb{R}^{2}$ as $u^{\varepsilon}(t)$ where it is understood that our setting is in two dimensions. We next introduce the standard spaces and notations for the deterministic Navier-Stokes equation (NSE) (cf. [37]). Let
\begin{eqnarray*}
V &:=& \left\{u\in H_{0}^{1}(D;\mathbb{R}^{2}): \nabla \cdot u=0\right\}, \\
H &:=& \left\{u\in L^{2}(D;\mathbb{R}^{2}): \nabla \cdot u = 0\right\},
\end{eqnarray*}
with norms,
\begin{equation*}
\|u\|_{V} = \left(\int_{D}\mid \nabla u \mid^{2} dx\right)^{1/2}, \hspace{.4cm} \text{and} \hspace{.3cm} \|u\|_{H}= \left(\int_{D} \mid u\mid^{2}dx\right)^{1/2},
\end{equation*}
respectively. Note that spaces $H$ and $V$ are the closures of the divergence free smooth compactly supported functions in $\|\cdot \|_{L^{2}}$ and $\|\cdot\|_{H^{1}}$, respectively. Here we will follow the conventional notation in denoting the norm in $H$ as $\mid\cdot\mid$ and the norm in $V$ as $\|\cdot\|$. For space $H$, divergence, $\nabla \cdot u=0$, is understood in distributional sense and $u\cdot \hat{n}\mid_{\partial \Omega}=0$ is well-defined (see Theorems 1.4 and 1.6 of [37]). Both $H$ and $V$ are Hilbert spaces, and in particular $V$ may be equipped with the inner product,
\begin{equation*}
(u,v)_{V}= \sum_{i,j=1}^{2} \int_{D} \partial_{i}u_{j}\partial_{i}v_{j}dx.
\end{equation*}
\hspace{.5cm} Letting $H'$ and $V'$ denote the dual spaces of $H$ and $V$, respectively, we identify $H$ with $H'$ by the Riesz representation theorem to obtain, $V\hookrightarrow H \equiv H' \hookrightarrow V'$, where the embeddings are dense and compact. Furthermore, using the Helmholtz-Leray projection, $P_{H}:L^{2}(D;\mathbb{R}^{2})\rightarrow H$, we define,
\begin{eqnarray*}
Au &:=& -P_{H}\Delta u, \hspace{.4cm} \forall u\in H^{2}(D; \mathbb{R}^{2}) \cap V,\\
B(u,v) &:=& P_{H} \left(\left(u\cdot \nabla\right)v\right), \hspace{.4cm} \forall u,v\in D(B)\subset V\times V,
\end{eqnarray*}
where $A$ is a positive-definite, self-adjoint operator referred to as the Stokes operator. Operators $A$ and $B$ may be defined explicitly as follows,
\begin{eqnarray}
(Au,v) &:=& \sum_{i,j=1}^{2} \int_{D} \partial_{i}u_{j}\partial_{i}v_{j} dx,\\
\left(B(u,v),w\right) &:=& \sum_{i,j=1}^{2} \int_{D} u_{i}\partial_{i}v_{j}w_{j}dx =: b(u,v,w).
\end{eqnarray}
One may observe that $b(u,v,w)= - b(u,w,v)$ leading to $b(u,v,v)=0$. For estimates derived in the rest of the paper we have the following inequalities derived in [35, 38],
\begin{eqnarray}
\mid b(u,v,w)\mid &\leq& 2\|u\|^{1/2} \cdot  \mid u\mid^{1/2} \cdot \|v\|^{1/2} \cdot \mid v\mid^{1/2} \cdot \|w\|, \label{(2.3)}\\
\mid b(u,u,v)\mid &\leq& \frac{1}{2} \|u\|^{2} + c\|v\|_{L^{4}}^{4} \cdot \mid u\mid^{2}, \label{(2.4)}\\
\mid\left(B(u)-B(v),u-v\right)\mid &\leq& \frac{1}{2} \|u-v\|^{2} + c\mid u-v \mid^{2} \|v\|_{L^{4}}^{4}. \label{(2.5)}
\end{eqnarray}
\hspace{.5cm} Projecting system (1.1) onto the divergence free vector-fields by $P_{H}$, we obtain the preferred abstract version of SNSE as follows,
\begin{equation}\label{(2.6)}
du^{\varepsilon}(t) + Au^{\varepsilon}(t)dt + B(u^{\varepsilon}(t))dt = f(t)dt + \sqrt{\varepsilon}\sigma(t,u^{\varepsilon}(t))dW(t),
\end{equation}
in a probability space, $(\Omega, \mathcal{F}, P)$, where $W(\cdot)$ is an $H$ valued $\{\mathcal{F}_{t}\}_{t\geq 0}$-adapted Q-Wiener process, that may be written as,
\begin{equation*}
W(t):= \sum_{j=1}^{\infty} \sqrt{\lambda_{j}}e_{j}\beta_{j}(t),
\end{equation*}
 for an infinite sequence of independent, standard one dimensional $\{\mathcal{F}_{t}\}_{t\geq 0}$ Brownian motions and a complete orthonormal system $\{e_{j}\}_{j=1}^{\infty}$ in $H$ satisfying $Qe_{j}= \lambda_{j}e_{j}$, where $\lambda_{j}$ is the $j^{\text{th}}$ eigenvalue of the covariance operator $Q$. Furthermore, we define the Hilbert space, $H_{0}:= Q^{1/2}H$, with inner product,
\begin{equation*}
(u,v)_{0}= \left(Q^{-1/2}u, Q^{-1/2}v\right),
\end{equation*}
for all $u,v\in H_{0}$, where the embedding of $H_{0}$ in $H$ is Hilbert-Schmidt. We let $L_{Q}(H_{0}:H)$ be the space of linear operators $S$ such that $SQ^{1/2}$ is a Hilbert-Schmidt operator from $H$ to $H$, with norm, $\|S\|_{L_{Q}}:= \sqrt{tr(SQS^{*})}$. For more background on the Navier-Stokes equations in the deterministic setting we recommend [33, 34, 37]. We now state the assumptions required for our results. \\
\hspace{.5cm} Assumption (H$_{1}$): $f\in L^{4}(0,T;V')$ i.e. there exists a positive constant, $K_{0}$, such that
\begin{equation*}
\int_{0}^{T} \|f(s)\|_{V'}^{4}ds <K_{0},
\end{equation*}
and the function $\sigma:[0,T]\times V\rightarrow L_{Q}(H_{0};H)$ is bounded, satisfies the linear growth condition and is Lipschitz continuous. That is, for all $u,v\in V$ and all $t\in [0,T]$:
\begin{eqnarray}
&&\|\sigma(t,u)\|_{L_{Q}}\leq K_{1}, \hspace{.4cm} \|\sigma(t,u)\|_{L_{Q}}^{2}\leq K_{2}(1+\|u\|^{2}),\nonumber\\
&& \text{and} \hspace{.4cm} \|\sigma(t,u)-\sigma(t,v)\|_{L_{Q}}\leq K_{3}\|u-v\|. \label{(2.7)}
\end{eqnarray}
\hspace{.5cm} Assumption (H$_{2}$): suppose Assumption (H$_{1}$) holds and additionally suppose,
\begin{equation}\label{(2.8)}
\|\text{curl}\sigma(t,u)\|_{L_{Q}}^{2} \leq \widetilde{K}_{0} + \widetilde{K}_{1}\|u\|^{2},
\end{equation}
for $u\in D(A)$. For our results the following estimates achieved in [35, 38] are frequently used.
\begin{lemma}(Proposition 2.3 in [35] and Proposition 3.1 in [38]). \label{Lemm1}
If $u^{\varepsilon}(t)$ is the solution of SNSE \eqref{(2.6)}, then under Assumption (H$_{1}$), for any $\varepsilon < \frac{1}{2K_{0}^{2}} \wedge \frac{1}{4K_{0}} \wedge \frac{1}{2K_{2}}$,
\begin{eqnarray}
&&\mathbb{E}\left(\sup_{0\leq t\leq T} \mid u^{\varepsilon}(t)\mid^{2} + \int_{0}^{T}\|u^{\varepsilon}(s)\|^{2}ds\right)\leq K_{4}, \label{(2.9)}\\
&& \mathbb{E}\left(\sup_{0\leq t\leq T} \mid u^{\varepsilon}(t)\mid^{4} + \int_{0}^{T} \mid u^{\varepsilon}(s)\mid^{2}\|u^{\varepsilon}(s)\|^{2}ds\right)\leq K_{5},\label{(2.10)}\\
&&\sup_{0\leq t\leq T} \mid u^{0}(t)\mid^{2} + \int_{0}^{T} \|u^{0}(s)\|^{2}ds\leq K_{6}, \label{(2.11)}\\
&&\int_{0}^{T}\|u^{0}(s)\|^{4}_{L^{4}}ds\leq \sup_{0\leq t\leq T} \mid u^{0}(t)\mid^{2}\int_{0}^{T} \|u^{0}(s)\|^{2}ds \leq K_{7}, \label{(2.12)}\\
&&\mathbb{E} \left(\sup_{0\leq t\leq T} \mid u^{\varepsilon}(t)-u^{0}(t)\mid^{2} + \int_{0}^{T} \|u^{\varepsilon}(s)-u^{0}(s)\|^{2}ds\right) \leq \varepsilon K_{8}, \label{(2.13)}
\end{eqnarray}
where each constant above depends on $T$ and $K_{0}$.
\end{lemma}
We say that a family $\{u^{\varepsilon}(\cdot)\}_{\varepsilon>0}$ satisfies the moderate deviation principle, if the family $\{v^{\varepsilon}(\cdot)\}_{\varepsilon>0}$ defined as $v^{\varepsilon}(t):= (a(\varepsilon)/\sqrt{\varepsilon})(u^{\varepsilon}(t)-u^{0}(t))$ obeys the large deviation principle where conditions on $a(\varepsilon)$ are $a(\varepsilon)>0$ and $a(\varepsilon)/\sqrt{\varepsilon} \rightarrow \infty$ as $\varepsilon$ tends to zero. This ensures that the rate of decay of moderate deviation given by $a(\varepsilon)$ is at a slower speed than the rate of decay for large deviation given by $\sqrt{\varepsilon}$. Let $\mathcal{H}_{0}$ be the Cameron-Martin space consisting of absolutely continuous functions, $h:[0,T]\rightarrow H_{0}$ such that $\int_{0}^{T} \mid h(s)\mid_{0}^{2}ds<\infty$ and denote $S_{N}:= \{h\in \mathcal{H}_{0}: \int_{0}^{T} \mid h(s)\mid_{0}^{2}ds\leq N\}$. Then for every $h\in S_{N}$, the controlled PDEs for $v^{\varepsilon}(t)$, also referred to as the skeleton equation, is given by,
\begin{equation}\label{(2.14)}
dX^{h}(t) + AX^{h}(t) = -B(X^{h}(t), u^{0}(t))dt - B(u^{0}(t), X^{h}(t))dt + \sigma(t, u^{0}(t))h(t)dt,
\end{equation}
where $X^{h}(0)=0$ and for which there exists a unique solution, denoted as $\Gamma^{0}(\int_{0}^{\cdot}h(s)ds)$, in $\mathcal{C}([0,T];H)\cap L^{2}(0,T;V)$ (cf. [38]). Let,
\begin{equation}\label{(2.15)}
\varepsilon_{0}:= \min\left\{1, \frac{1}{2K_{0}^{2}}, \frac{1}{4K_{0}}, \frac{1}{2K_{2}}, \frac{1}{78K_{9}}\right\},
\end{equation}
where, $K_{9}$ is a positive constant, introduced later in the proof of the main results.
\begin{theorem}\label{Them1}
Family $\{u^{\varepsilon}(\cdot)\}_{\varepsilon\in (0,\varepsilon_{0})}$ satisfies the moderate deviation principle in $\mathcal{C}([0,T];H)\cap L^{2}(0,T;V)$ with speed $a(\varepsilon)^{2}$ and rate function,
\begin{align}\label{(2.16)}
I(v) = \left\{\begin{array}{ll}
\frac{1}{2}\inf \int_{0}^{T} \mid h(s)\mid_{0}^{2}ds, & \hspace{.2cm}\text{for} \hspace{.1cm} v= \Gamma^{0}(\int_{0}^{\cdot}h(s)ds), h\in \mathcal{H}_{0},\\
\infty &\hspace{.2cm}\text{otherwise}.
\end{array}\right.
\end{align}
\end{theorem}
For the above theorem, we first prove the result for $a(\varepsilon)= 1/\sqrt{2\log \log \frac{1}{\varepsilon}}$, which satisfies the required conditions on $a(\varepsilon)$ and gives the process,
\begin{equation}\label{(2.17)}
Z^{\varepsilon}(t):= \frac{1}{\sqrt{2\varepsilon \log \log \frac{1}{\varepsilon}}}\left(u^{\varepsilon}(t)-u^{0}(t)\right).
\end{equation}
Namely,
\begin{eqnarray}\label{(2.18)}
Z^{\varepsilon}(t) &=& -\int_{0}^{t}AZ^{\varepsilon}(s)ds - \int_{0}^{t} B\left(Z^{\varepsilon}(s), u^{\varepsilon}(s)\right)ds - \int_{0}^{t} B\left(u^{0}(s), Z^{\varepsilon}(s)\right)ds\nonumber\\
&& + \frac{1}{\sqrt{2 \log \log \frac{1}{\varepsilon}}}\int_{0}^{t} \widetilde{\sigma}(s, Z^{\varepsilon}(s))dW(s),
\end{eqnarray}
with
\begin{equation*}
\widetilde{\sigma}(t, Z^{\varepsilon}(t)):= \sigma\left(t, \sqrt{2\varepsilon \log \log \frac{1}{\varepsilon}}Z^{\varepsilon}(t) + u^{0}(t)\right).
\end{equation*}
Note that based on assumptions in \eqref{(2.7)} and \eqref{(2.8)}, the following conditions hold,
\begin{eqnarray}
&& \|\widetilde{\sigma}(t,Z^{\varepsilon}(t))\|_{L_{Q}}^{2} \leq K_{9}\left(1+ 4\varepsilon \log \log \frac{1}{\varepsilon}\|Z^{\varepsilon}(t)\|^{2} + 2\|u^{0}(t)\|^{2}\right),\label{(2.19)}\\
&& \|\widetilde{\sigma}(t,Z_{1}^{\varepsilon}(t))- \widetilde{\sigma}(t, Z_{2}^{\varepsilon}(t))\|_{L_{Q}} \leq K_{10} \sqrt{2\varepsilon  \log \log \frac{1}{\varepsilon}}\|Z_{1}^{\varepsilon}(t)-Z^{\varepsilon}_{2}(t)\|, \label{(2.20)}\\
&&\|\text{curl} \widetilde{\sigma}(t,Z^{\varepsilon}(t))\|_{L_{Q}}^{2}\leq \widetilde{K}_{2} + \widetilde{K}_{3}\sqrt{2\varepsilon \log \log \frac{1}{\varepsilon}}(\|Z^{\varepsilon}(t)\| + \|u^{0}(t)\|). \label{(2.21)}
\end{eqnarray}
The following is the definition of Strassen's Compact LIL and the result achieved in this paper. For more information and similar results on this type of LIL we recommend [4, 17, 22, 39].
\begin{definition}
A class of functions, $\mathcal{F}$ satisfies Strassen's Compact LIL with respect to an i.i.d. sequence of random variables, $\{X_{j}\}_{j\geq 1}$ if there exists a compact set $J$ in $\ell_{\infty}(\mathcal{F})$ such that $\{X_{j}\}_{j\geq 1}$ is a.s. relatively compact and its limit set is precisely $J$.
\end{definition}
\begin{theorem}\label{Them2}
Family $\{Z^{\varepsilon}(\cdot)\}_{\varepsilon\in (0,\varepsilon_{0})}$ is relatively compact in $\mathcal{C}([0,T];H)\cap L^{2}(0,T;V)$ and its set of limit points is exactly, \\
$L:=\{g\in \mathcal{C}([0,T];H)\cap L^{2}(0,T;V) : I(g)\leq \frac{N}{2}\}$, where $I(g)$ is given by \eqref{(2.16)}.
\end{theorem}
\section{Moderate Deviations}
\hspace{.5cm} We prove the moderate deviation principle for $\{u^{\varepsilon}(\cdot)\}_{\varepsilon \in (0,\varepsilon_{0})}$ by establishing the large deviation principle for $\{Z^{\varepsilon}(\cdot)\}_{\varepsilon\in (0,\varepsilon_{0})}$. The Azencott method implemented here may be described as follows. Consider two families of random variables $\{Y_{1}^{\varepsilon}\}_{\varepsilon>0}$, $\{Y_{2}^{\varepsilon}\}_{\varepsilon>0}$ taking values in Polish spaces $\mathcal{E}_{1}, \mathcal{E}_{2}$, respectively with the corresponding metrics $d_{1},d_{2}$. Suppose $\{Y_{1}^{\varepsilon}\}_{\varepsilon>0}$ satisfies the large deviation principle with rate function $\widetilde{I}(g)$ where $g\in \mathcal{E}_{1}$. Let $a>0$ be fixed and $\Phi: \{\widetilde{I}<\infty\} \rightarrow \mathcal{E}_{2}$. For any $R>0, \rho>0$, suppose there exists $\overline{\varepsilon}>0$ and $\eta>0$ such that for all $0<\varepsilon\leq \overline{\varepsilon}$, the following inequality,
\begin{equation}\label{(3.1)}
P\left(d_{2}(Y_{2}^{\varepsilon}, \Phi(f)) \geq \rho, d_{1}\left(Y_{1}^{\varepsilon},f\right)<\eta\right)\leq \exp\left(-\frac{R}{\varepsilon^2}\right),
\end{equation}
referred to as the Freidlin-Wentzell inequality, holds for any $g\in \mathcal{E}_{1}$ with $\widetilde{I}(g)\leq a$. In addition, suppose the map $\Phi(\cdot)$ is continuous with respect to the topology of $\mathcal{E}_{1}$ when restricted to compact sets $\{\widetilde{I}\leq a\}_{a>0}$ for every positive constant $a$. Then $\{Y_{2}^{\varepsilon}\}_{\varepsilon>0}$ also satisfies the large deviation principle with rate function, $I(h):= \inf\{\widetilde{I}(g): \Phi(g)=h\}$ for any $h\in \mathcal{E}_{2}$. In the setting of the stochastic PDEs, the map $\Phi(\cdot)$ is the unique solution of the skeleton equation, which in our case is given by \eqref{(2.14)}. We let $Y_{1}^{\varepsilon}:= (1/\sqrt{2\log \log \frac{1}{\varepsilon}})W$, which by an extension of the result of the Schilder's theorem to Q-Wiener process, is known to satisfy the large deviation principle with a good rate function (for a proof see Theorem 3.1 in Chapter 6 of [2]). Thus, it is sufficient to prove that $h\mapsto X^{h}$ is continuous and achieve inequality \eqref{(3.1)}. For examples of large deviations results using this technique for stochastic PDEs we refer the reader to [7, 13, 15, 28]. We begin by establishing the continuity of the map $h\mapsto X^{h}(t)$. For simplicity of notation, we denote $L^{\infty}([0,T];H_{0})$ as $L_{H_{0}}^{\infty}$ and for the rest of the article we let the norm in $\mathcal{C}([0,T];H)\cap L^{2}(0,T;V)$ be denoted as $\mathcal{E}(T)$. Namely,
\begin{equation*}
\|u(t)\|_{\mathcal{E}(T)} := \left(\sup_{0\leq t\leq T} \mid u(t)\mid^{2} + \int_{0}^{T} \|u(s)\|^{2}ds\right)^{1/2}.
\end{equation*}
\begin{lemma}\label{lemm2}
For every $h\in \mathcal{H}_{0}$ and $a>0$, the map $h\mapsto X^{h}$ is continuous in $\mathcal{C}([0,T];H)\cap L^{2}(0,T;V)$ with respect to the uniform convergence topology when restricted to the level set $\{I\leq a\}_{a>0}$.
\end{lemma}
\begin{proof}
Let $a>0$ and $h,k\in \mathcal{H}_{0}$ such that $\mid h\mid_{0}\vee \mid k\mid_{0} \leq a$, then we have using $b(u,v,v)=0$,
\begin{eqnarray*}
&&\mid X^{h}(t)-X^{k}(t)\mid^{2} + 2\int_{0}^{t} \|X^{h}(s)-X^{k}(s)\|^{2}ds\\
&&= -2\int_{0}^{t} \left(B(X^{h}(s)-X^{k}(s), u^{0}(s)), X^{h}(s)-X^{k}(s)\right)ds\\
&& + 2\int_{0}^{t} \left(\sigma(s,u^{0}(s))(h(s)-k(s)), X^{h}(s)-X^{k}(s)\right)ds.
\end{eqnarray*}
By inequality \eqref{(2.4)} and Young's inequality along with Assumption (H$_{1}$) we obtain,
\begin{eqnarray}\label{(3.2)}
&&\mid X^{h}(t)-X^{k}(t)\mid^{2} + 2\int_{0}^{t} \|X^{h}(s)-X^{k}(s)\|^{2}ds\nonumber \\
&& \leq \int_{0}^{t} \|X^{h}(s)-X^{k}(s)\|^{2}ds + 2 \int_{0}^{t} c \|u^{0}(s)\|_{L^{4}}^{4} \mid X^{h}(s)-X^{k}(s)\mid^{2}ds\nonumber\\
&& + 2\|h(s)-k(s)\|_{L^{\infty}_{H_{0}}}^{2} + \frac{1}{2} \int_{0}^{t} \|\sigma(s,u^{0}(s))\|_{L_{Q}}^{2} \mid X^{h}(s)-X^{k}(s)\mid^{2}ds\nonumber\\
&&\leq \int_{0}^{t} \left(2c\|u^{0}(s)\|_{L^{4}}^{4} + \frac{K_{2}}{2} + \frac{K_{2}}{2}\|u^{0}(s)\|^{2}\right)\mid X^{h}(s)-X^{k}(s)\mid^{2}ds\nonumber\\
&& + \int_{0}^{t} \|X^{h}(s)-X^{k}(s)\|^{2}ds + 2 \|h(s)-k(s)\|_{L_{H_{0}}^{\infty}}^{2},
\end{eqnarray}
which by Gronwall's inequality yields,
\begin{eqnarray}\label{(3.3)}
&&\mid X^{h}(t)-X^{k}(t)\mid^{2}\\
&&\leq 2 \|h(s)-k(s)\|_{L_{H_{0}}^{\infty}}^{2}\exp\left(\int_{0}^{T} 2c\|u^{0}(s)\|_{L^{4}}^{4} + \frac{K_{2}}{2} + \frac{K_{2}}{2}\|u^{0}(s)\|^{2}ds\right).\nonumber
\end{eqnarray}
Noting inequalities \eqref{(2.11)} and \eqref{(2.12)} and applying \eqref{(3.3)} on the right hand side of \eqref{(3.2)}, we achieve the continuity of the map $h\mapsto X^{h}$.
\end{proof}
\hspace{.5cm} Now we focus on obtaining the Freidlin-Wentzell inequality \eqref{(3.1)}, which for our model is, for any $R>0$ and $\rho>0$, there exists $\eta>0$ and $\overline{\varepsilon}>0$ such that for all $0<\varepsilon\leq \overline{\varepsilon}$,
\begin{eqnarray}\label{(3.4)}
&&P\left(\|Z^{\varepsilon}(t)-X^{h}(t)\|_{\mathcal{E}(T)} \geq \rho, \left\| \frac{1}{\sqrt{2\log \log \frac{1}{\varepsilon}}} W-h\right\|_{L_{{H}}^{\infty}} <\eta\right)\nonumber\\
&&\leq \exp\left(-2R\log \log \frac{1}{\varepsilon}\right),
\end{eqnarray}
where as noted earlier, $\{(1/\sqrt{2\log \log \frac{1}{\varepsilon}})W\}_{\varepsilon>0}$ satisfies the large deviation principle. By Girsanov's transformation theorem, inequality \eqref{(3.4)} is implied by the following inequality (see [13, 28, 32] for more details),
\begin{eqnarray}\label{(3.5)}
&&P\left(\|\widetilde{Z}^{\varepsilon}(t)-X^{h}(t)\|_{\mathcal{E}(T)}\geq \rho, \left\| \frac{1}{\sqrt{2\log \log \frac{1}{\varepsilon}}} W\right\|_{L_{{H}}^{\infty}} <\eta\right) \nonumber\\
&&\leq \exp\left(-2R\log \log \frac{1}{\varepsilon}\right),
\end{eqnarray}
where,
\begin{eqnarray*}
\widetilde{Z}^{\varepsilon}(t)&=& - \int_{0}^{t} A\widetilde{Z}^{\varepsilon}(s)ds - \int_{0}^{t}B(u^{\varepsilon}(s), \widetilde{Z}^{\varepsilon}(s))ds - \int_{0}^{t} B(\widetilde{Z}^{\varepsilon}(s), u^{0}(s))ds\\
&& + \frac{1}{\sqrt{2\log \log \frac{1}{\varepsilon}}} \int_{0}^{t} \widetilde{\sigma}(s, \widetilde{Z}^{\varepsilon}(s))dW(s) + \int_{0}^{t} \widetilde{\sigma}(s, \widetilde{Z}^{\varepsilon}(s))h(s)ds.
\end{eqnarray*}
We now apply a time discretization on $\widetilde{Z}^{\varepsilon}(t)$ by letting $\Delta_{j}^{n}:= [t_{j}^{n}, t_{j+1}^{n})$, where, $n\in \mathbb{N}\backslash \{0\}, j= 0,1,2,..,2^n -1$ and $t_{j}^{n}= (Tj)/2^{n}$. Then to achieve \eqref{(3.5)}, it is sufficient to prove that there exists $\beta>0, n_{0}\in \mathbb{N}\backslash \{0\}$ and $\overline{\varepsilon}>0$ such that for all $n\geq n_{0}, \varepsilon \in (0,\overline{\varepsilon})$,
\begin{equation}
P\left(\|\widetilde{Z}^{\varepsilon}(t)- \widetilde{Z}^{\varepsilon}(t_{i}^{n})\|_{\mathcal{E}(T)} >\beta\right)\leq \exp\left(-2R\log \log \frac{1}{\varepsilon}\right),\label{(3.6)}
\end{equation}
and
\begin{flalign}\label{(3.7)}
&P\left(\|\widetilde{Z}^{\varepsilon}(t)-X^{h}(t)\|_{\mathcal{E}(T)}\geq \rho, \left\|\frac{1}{\sqrt{2\log \log \frac{1}{\varepsilon}}} W\right\|_{L_{H}^{\infty}}<\eta, \right. \nonumber&\\
&\hspace{2.5cm}\left. \|\widetilde{Z}^{\varepsilon}(t)-\widetilde{Z}^{\varepsilon}(t_{i}^{n})\|_{\mathcal{E}(T)}\leq \beta\right)\leq \exp\left(-2R\log \log \frac{1}{\varepsilon}\right).&
\end{flalign}
For this purpose the following lemmas are proved and applied. For better presentation, their proofs are given in the Appendix. Recall $\varepsilon_{0}$ defined in \eqref{(2.15)}. Let $p\geq 1$ and
\begin{equation*}
\varepsilon_{1}:= \min\left\{1, \frac{1}{2K_{0}^{2}}, \frac{1}{4K_{0}}, \frac{1}{2K_{2}}, \frac{1}{36K_{9}}\right\}, \hspace{.4cm}
\varepsilon_{2}:= \min\left\{\varepsilon_{1}, \frac{1}{78K_{9}}, \frac{1}{K_{9}(36p+2)}\right\}.
\end{equation*}
\begin{lemma}\label{Lemm3}
If $\varepsilon\in (0,\varepsilon_{1})$,
\begin{equation}\label{(3.8)}
\mathbb{E}\sup_{0\leq s\leq t} \mid \widetilde{Z}^{\varepsilon}(s)\mid^{2} + \int_{0}^{t} \mathbb{E}\|\widetilde{Z}^{\varepsilon}(s)\|^{2}ds \leq \widetilde{M}_{1}(T,\varepsilon),
\end{equation}
where $\widetilde{M}_{1}(T,\varepsilon)$ is a positive constant.
\end{lemma}
\begin{lemma}\label{Lemm4}
For $0<\varepsilon<\varepsilon_{1}\wedge\frac{1}{78K_{9}}$,
\begin{equation}\label{(3.9)}
\mathbb{E} \sup_{0\leq s\leq T} \mid \widetilde{Z}^{\varepsilon}(s)\mid^{4} + \mathbb{E}\int_{0}^{T} \mid\widetilde{Z}^{\varepsilon}(s)\mid^{2}\|\widetilde{Z}^{\varepsilon}(s)\|^{2}ds \leq \widetilde{M}_{2}(T,\varepsilon),
\end{equation}
and for $\varepsilon\in (0,\varepsilon_{2})$ and any $p\geq 1$,
\begin{equation}\label{(3.10)}
\mathbb{E}\sup_{0\leq s\leq t} \mid \widetilde{Z}^{\varepsilon}(s)\mid^{2p} + \mathbb{E}\int_{0}^{T} \mid\widetilde{Z}^{\varepsilon}(s)\mid^{2(p-1)} \|\widetilde{Z}^{\varepsilon}(s)\|^{2}ds \leq \widetilde{M}_{p}(T,\varepsilon),
\end{equation}
where $\widetilde{M}_{2}(T,\varepsilon)$ and $\widetilde{M}_{p}(T,\varepsilon)$ are positive constants.
\end{lemma}
We note that in the proof of the main results, \eqref{(3.10)} is used only for $p=1$ and $p=2$. Thus, we choose $\varepsilon_{0}$ given in \eqref{(2.15)} to apply the above estimates and those offered by Lemma \ref{Lemm1}. With the same set of techniques used to prove the above estimates, we may derive the following bounds. For the proofs of inequalities \eqref{(3.11)} and \eqref{(3.12)} we refer the reader to Lemma 4.2 in [26] and Proposition 4.4 in [38], respectively.
\begin{lemma}\label{Lemm5}
For any $p\geq1$ if $\varepsilon < \frac{2}{1+2p}$,
\begin{equation}\label{(3.11)}
\mathbb{E}\sup_{0\leq s\leq T} \mid u^{\varepsilon}(s)\mid^{2p} + \mathbb{E} \int_{0}^{T} \mid u^{\varepsilon}(s)\mid^{2(p-1)}\|u^{\varepsilon}(s)\|^{2}ds\leq \widetilde{N}_{p}(\varepsilon, T, u^{\varepsilon}(0)),
\end{equation}
and for any $N>0$ and $h\in \mathcal{H}_{0}$,
\begin{equation}\label{(3.12)}
\sup_{h\in S_{N}} \left(\sup_{0\leq t\leq T} \mid X^{h}(t)\mid^{2} + \int_{0}^{T}\|X^{h}(s)\|^{2}ds\right)\leq \widetilde{N},
\end{equation}
where $\widetilde{N}$ depends on $T$ and $N$.
\end{lemma}
In addition, we use the proposition below established in [5] for 2D stochastic Navier Stokes equation having viscosity, $\nu>0$, given as,
\begin{eqnarray*}
du_{h}^{\nu}(t)&=& - \left(\nu Au_{h}^{\nu}(t) + B\left(u_{h}^{\nu}(t), u_{h}^{\nu}(t)\right)\right)dt \\ &&+\widetilde{\sigma}_{\nu}(t,u_{h}^{\nu}(t))h(t)dt + \sqrt{\nu} \sigma_{\nu}(t,u_{h}^{\nu}(t))dW(t),\\
u_{h}^{\nu}(0)&=& \eta.
\end{eqnarray*}
\begin{proposition}(Proposition 2.2 in [5]). Let $p\geq 2,$ $\mathbb{E}\|\eta\|^{2p}<\infty$ and suppose Assumption (H$_{2}$) holds. Then given $M>0$, there exists a positive constant $C_{2}(p,M)$ such that for $\nu\in (0,\nu_{0}]$ and $h\in S_{M}$,
\begin{equation}\label{(3.13)}
\mathbb{E}\left(\sup_{0\leq t\leq T} \|u_{h}^{\nu}(t)\|^{2p} + \nu \int_{0}^{T} \mid Au_{h}^{\nu}(s)\mid^{2}ds \right) \leq \nu C_{2}(p, M)(1+ \mathbb{E} \|\eta\|^{2p}).
\end{equation}
\end{proposition}
The main idea in their proof is to apply the curl to the solution and then use a stochastic Gronwall inequality offered by Lemma A.1 of [16]. By the same reasoning this inequality may be proved for $\{\widetilde{Z}^{\varepsilon}_{.}\}_{\varepsilon>0}$ with the difference of having one extra nonlinear term,
\begin{equation*}
\int_{0}^{t} \left(\text{curl}B\left(\widetilde{Z}^{\varepsilon}(s), u^{0}(s)\right), \text{curl} \widetilde{Z}^{\varepsilon}(s)\right)ds,
\end{equation*}
and bounding this term may be achieved by noting that
\begin{equation*}
\text{curl}B\left(\widetilde{Z}^{\varepsilon}(s), u^{0}(s)\right) = B(\widetilde{Z}^{\varepsilon}(s), \text{curl}u^{0}(s))
- B(\text{curl}\widetilde{Z}^{\varepsilon}(s), u^{0}(s)),
\end{equation*}
and applying inequality \eqref{(2.3)}. Then analogous bounds for $u^{\varepsilon}(t)$ and $\widetilde{Z}^{\varepsilon}(t)$ are,
\begin{flalign}\label{(3.14)}
&\mathbb{E}\left(\sup_{0\leq t\leq T} \|u^{\varepsilon}(t)\|^{2p} + \int_{0}^{T} \mid Au^{\varepsilon}(s)\mid^{2}ds\right)
\leq \varepsilon C_{2}(p,M)\left(1+\|u^{\varepsilon}(0)\|^{2p}\right)&\\ \nonumber
&=:\widetilde{C}_{1}(2p,\varepsilon),&
\end{flalign}
and
\begin{equation}\label{(3.15)}
\mathbb{E}\left(\sup_{0\leq t\leq T} \|\widetilde{Z}^{\varepsilon}(t)\|^{2p} + \int_{0}^{T} \mid A\widetilde{Z}^{\varepsilon}(s)\mid^{2}ds\right)
\leq \frac{1}{2\log \log \frac{1}{\varepsilon}} C_{2}(p,M) =: \widetilde{C}_{2}(2p,\varepsilon),
\end{equation}
respectively. Moreover, similar to the proof of Theorem 3.1 of [5], the following may be established,
\begin{eqnarray}
\sup_{0\leq t\leq T} \|u^{0}(t)\|^{2p} &\leq& C(M)(1+ \|u^{0}(0)\|^{2p}):= \widetilde{C}_{3}(2p),\label{(3.16)}\\
\sup_{0\leq t\leq T} \|X^{h}(t)\|^{2p} &\leq& C(M)(1+\|X^{h}(0)\|^{2p}) := \widetilde{C}_{4}(2p).\label{(3.17)}
\end{eqnarray}
\begin{lemma}\label{Lemm6} For every positive constant $R>0$ and $\beta>0$ there exists $n_{0}\in \mathbb{N}$ such that for all $n\geq n_{0}$ and $\varepsilon\in (0,\varepsilon_{0})$,
\begin{equation}\label{(3.18)}
P\left(\|\widetilde{Z}^{\varepsilon}(t)-\widetilde{Z}^{\varepsilon}(t_{i}^{n})\|_{\mathcal{E}(T)} >\beta\right) \leq \exp\left(-2R\log \log \frac{1}{\varepsilon}\right).
\end{equation}
\end{lemma}
\begin{proof}
Using the time discretization introduced earlier, we apply It$\hat{o}$'s formula to find for $t\in [t_{i}^{n}, t_{i+1}^{n}), s_{i}^{n}\in [0,t_{i}^{n}]$, and stopping time, \\
$\tau_{N}:= \inf\{t: \|\widetilde{Z}^{\varepsilon}(t)-\widetilde{Z}^{\varepsilon}(t_{i}^{n})\|_{\mathcal{E}(t)}^{2}>N\}$, where $N\in \mathbb{N}$,
\begin{eqnarray*}
&&\mid \widetilde{Z}^{\varepsilon}(t\wedge \tau_{N}) - \widetilde{Z}^{\varepsilon}(t_{i}^{n} \wedge \tau_{N})\mid^{2} + 2\int_{t_{i}^{n}\wedge \tau_{N}}^{t\wedge \tau_{N}} \|\widetilde{Z}^{\varepsilon}(s)-\widetilde{Z}^{\varepsilon}(s_{i}^{n})\|^{2}ds\\
&=&-2\int_{t_{i}^{n}\wedge \tau_{N}}^{t\wedge \tau_{N}} \left(A\widetilde{Z}^{\varepsilon}(s_{i}^{n}), \widetilde{Z}^{\varepsilon}(s)-\widetilde{Z}^{\varepsilon}(s_{i}^{n})\right)ds\\
&&-2\int_{t_{i}^{n} \wedge \tau_{N}}^{t\wedge \tau_{N}} \left(B(u^{\varepsilon}(s), \widetilde{Z}^{\varepsilon}(s)), \widetilde{Z}^{\varepsilon}(s)- \widetilde{Z}^{\varepsilon}(s_{i}^{n})\right)ds\\
&&-2\int_{t_{i}^{n}\wedge \tau_{N}}^{t\wedge \tau_{N}} \left(B(\widetilde{Z}^{\varepsilon}(s), u^{0}(s)), \widetilde{Z}^{\varepsilon}(s)-\widetilde{Z}^{\varepsilon}(s_{i}^{n})\right)ds\\
&&+\frac{2}{\sqrt{2\log \log \frac{1}{\varepsilon}}} \int_{t_{i}^{n}\wedge \tau_{N}}^{t\wedge \tau_{N}} \left(\widetilde{\sigma}(s,\widetilde{Z}^{\varepsilon}(s))dW(s), \widetilde{Z}^{\varepsilon}(s)-\widetilde{Z}^{\varepsilon}(s_{i}^{n})\right)\\
&&+ 2 \int_{t_{i}^{n}\wedge \tau_{N}}^{t\wedge \tau_{N}} \left(\widetilde{\sigma}(s, \widetilde{Z}^{\varepsilon}(s))h(s), \widetilde{Z}^{\varepsilon}(s)-\widetilde{Z}^{\varepsilon}(s_{i}^{n})\right)ds\\
&&+\frac{1}{2\log \log \frac{1}{\varepsilon}}\int_{t_{i}^{n}\wedge \tau_{N}}^{t\wedge \tau_{N}} \|\widetilde{\sigma}(s,\widetilde{Z}^{\varepsilon}(s))\|^{2}_{L_{Q}}ds\\
&=& I_{0}(t_{i}^{n}\wedge \tau_{N}, t\wedge \tau_{N}) + I_{1}(t_{i}^{n}\wedge \tau_N, t\wedge \tau_{N}) + I_{2}(t_{i}^{n} \wedge \tau_{N},t\wedge \tau_{N})\\
&&+ I_{3}(t_{i}^{n}\wedge \tau_N, t\wedge \tau_{N})+ I_{4}(t_{i}^{n}\wedge \tau_N, t\wedge \tau_{N})+I_{5}(t_{i}^{n}\wedge \tau_N, t\wedge \tau_{N}).
\end{eqnarray*}
Similar to the bounds in the Appendix derived for the proof of Lemma \ref{Lemm3}, we will take the supremum up to time $t\wedge \tau_{N}$ and then expectation and determine the following bounds for $\mathbb{E}\mid I_{j}(t_{i}^{n} \wedge \tau_{N}, t\wedge \tau_{N})\mid$ for $j=0,1,...,5, j\neq 3,$ for which the bound in \eqref{(3.15)} is used. For $I_{0}(t_{i}^{n}\wedge \tau_{N}, t\wedge \tau_{N})$, by applying the Young's inequality, we obtain,
\begin{flalign*}
&\mathbb{E}\mid I_{0}(t_{i}^{n}\wedge \tau_N, t\wedge \tau_{N})\mid \leq 2\mathbb{E}\int_{t_{i}^{n}\wedge \tau_{N}}^{t\wedge \tau_{N}} (A\widetilde{Z}^{\varepsilon}(s_{i}^{n}), \widetilde{Z}^{\varepsilon}(s))ds + 2\mathbb{E} \int_{t_{i}^{n}\wedge \tau_{N}}^{t\wedge \tau_{N}}\|\widetilde{Z}^{\varepsilon}(s_{i}^{n})\|^{2}ds&\\
&\hspace{.3cm} \leq 2\mathbb{E}\int_{t_{i}^{n}\wedge \tau_{N}}^{t\wedge \tau_{N}} \|\widetilde{Z}^{\varepsilon}(s_{i}^{n})\|\hspace{.3cm} \|\widetilde{Z}^{\varepsilon}(s)\|ds + 2\mathbb{E} \int_{t_{i}^{n}\wedge \tau_{N}}^{t\wedge \tau_{N}} \|\widetilde{Z}^{\varepsilon}(s_{i}^{n})\|^{2}ds&\\
&\hspace{.3cm} \leq 4 \mathbb{E}\int_{t_{i}^{n}\wedge \tau_{N}}^{t\wedge \tau_{N}} \|\widetilde{Z}^{\varepsilon}(s_{i}^{n})\|^{2} ds + \frac{1}{2} \mathbb{E} \int_{t_{i}^{n}\wedge \tau_{N}}^{t\wedge \tau_{N}} \|\widetilde{Z}^{\varepsilon}(s)\|^{2}&\\
&\hspace{.3cm} \leq \frac{9}{2} \widetilde{C}_{2}(2,\varepsilon) \mid t\wedge \tau_{N}-t_{i}^{n}\wedge \tau_{N}\mid.&
\end{flalign*}
Note that by Cauchy-Schwarz inequality,
\begin{equation*}
\mathbb{E}\int_{0}^{t} \|v(s)\|^{2} \mid v(s)\mid^{2}ds \leq \int_{0}^{t} \left(\mathbb{E}\|v(s)\|^{4}\right)^{\frac{1}{2}} \left(\mathbb{E}\mid v(s)\mid^{4}\right)^{\frac{1}{2}} ds.
\end{equation*}
Applying the above estimate along with \eqref{(2.3)} and Young's inequality, lead to,
\begin{eqnarray*}
&& \mathbb{E}\mid I_{1}(t_{i}^{n}\wedge \tau_{N}, t\wedge \tau_{N})\mid = 2\mid \mathbb{E} \int_{t_{i}^{n} \wedge \tau_{N}}^{t\wedge \tau_{N}} \left(B(u^{\varepsilon}(s), \widetilde{Z}^{\varepsilon}(s)), \widetilde{Z}^{\varepsilon}(s_{i}^{n})\right)ds\mid\\
&\leq& \frac{1}{2}\mathbb{E} \int_{t_{i}^{n} \wedge \tau_{N}}^{t\wedge \tau_{N}} \|u^{\varepsilon}(s)\| \hspace{.1cm} \mid u^{\varepsilon}(s)\mid ds+ 8\mathbb{E} \int_{t_{i}^{n} \wedge \tau_{N}}^{t\wedge \tau_{N}} \|\widetilde{Z}^{\varepsilon}(s)\| \hspace{.1cm} \mid \widetilde{Z}^{\varepsilon}(s)\mid \hspace{.1cm} \|\widetilde{Z}^{\varepsilon}(s_{i}^{n})\|^{2}ds\\
&\leq& \frac{1}{2} \mathbb{E} \int_{t_{i}^{n} \wedge \tau_{N}}^{t\wedge \tau_{N}} \|u^{\varepsilon}(s)\| \hspace{.1cm} \mid u^{\varepsilon}(s)\mid ds
+ 2 \mathbb{E} \int_{t_{i}^{n} \wedge \tau_{N}}^{t\wedge \tau_{N}} \|\widetilde{Z}^{\varepsilon} (s)\|^{2} \hspace{.1cm} \mid \widetilde{Z}^{\varepsilon}(s)\mid^{2}ds\\
&& + 8 \mathbb{E} \int_{t_{i}^{n} \wedge \tau_{N}}^{t\wedge \tau_{N}} \|\widetilde{Z}^{\varepsilon}(s_{i}^{n})\|^{4}ds\\
&\leq& \left(\frac{1}{2} \sqrt{K_{4}\widetilde{C}_{1}(2,\varepsilon)} + 2 \sqrt{\widetilde{C}_{2}(4,\varepsilon)\widetilde{M}_{2}(T,\varepsilon)} + 8 \widetilde{C}_{2}(4,\varepsilon)\right) \mid t\wedge \tau_{N}- t_{i}^{n}\wedge \tau_{N}\mid.
\end{eqnarray*}
Furthermore, by \eqref{(2.3)} and \eqref{(2.4)} we have,
\begin{flalign*}
&\mathbb{E}\mid I_{2}(t_{i}^{n} \wedge \tau_{N}, t\wedge \tau_{N})\mid&\\
&\hspace{.1cm} \leq 2 \mathbb{E} \mid\int_{t_{i}^{n} \wedge \tau_{N}}^{t\wedge \tau_{N}} b(\widetilde{Z}^{\varepsilon}(s), \widetilde{Z}^{\varepsilon}(s), u^{0}(s))ds +  \int_{t_{i}^{n} \wedge \tau_{N}}^{t\wedge \tau_{N}} b(\widetilde{Z}^{\varepsilon}(s), u^{0}(s), \widetilde{Z}^{\varepsilon}(s_{i}^{n}))ds\mid&\\
&\hspace{.1cm} \leq \mathbb{E} \int_{t_{i}^{n} \wedge \tau_{N}}^{t\wedge \tau_{N}} \|\widetilde{Z}^{\varepsilon}(s)\|^{2}ds
+ 2\mathbb{E} \int_{t_{i}^{n} \wedge \tau_{N}}^{t\wedge \tau_{N}} c \|u^{0}(s)\|_{L^{4}}^{4} \hspace{.1cm} \mid \widetilde{Z}^{\varepsilon}(s)\mid^{2}ds&\\
&\hspace{.1cm} + 2\mathbb{E} \int_{t_{i}^{n} \wedge \tau_{N}}^{t\wedge \tau_{N}} \|\widetilde{Z}^{\varepsilon}(s)\| \hspace{.1cm} \mid \widetilde{Z}^{\varepsilon}(s)\mid ds + \frac{1}{2} \int_{t_{i}^{n} \wedge \tau_{N}}^{t\wedge \tau_{N}} \|u^{0}(s)\|^{2} \hspace{.1cm} \mid u^{0}(s)\mid^{2}ds &\\
& \hspace{.1cm} + 2 \mathbb{E} \int_{t_{i}^{n} \wedge \tau_{N}}^{t\wedge \tau_{N}} \|\widetilde{Z}^{\varepsilon}(s_{i}^{n})\|^{4}ds&\\
&\hspace{.1cm} \leq \left(\widetilde{C}_{2}(2,\varepsilon) + 2K_{6}\widetilde{M}_{1}(T,\varepsilon) \widetilde{C}_{3}(2) + 2\sqrt{\widetilde{C}_{2}(2,\varepsilon) \widetilde{M}_{1}(T,\varepsilon)} + K_{6}\widetilde{C}_{3}(2) + \widetilde{C}_{2}(4,\varepsilon)\right)&\\
&\hspace{.1cm} \times \mid t\wedge \tau_{N} - t_{i}^{n} \wedge \tau_{N}\mid.&
\end{flalign*}
The Cauchy-Schwarz, H\"older's and Young's inequalities may again be invoked to obtain,
\begin{eqnarray*}
&&\mathbb{E}\mid I_{4}(t_{i}^{n}\wedge \tau_{N}, t\wedge \tau_{N})\mid \leq 2\mathbb{E} \int_{t_{i}^{n} \wedge \tau_{N}}^{t\wedge \tau_{N}}\|\widetilde{\sigma}(s,\widetilde{Z}^{\varepsilon}(s))\|_{L_{Q}} \hspace{.1cm} \mid h(s)\mid_{0} \hspace{.1cm} \mid \widetilde{Z}^{\varepsilon}(s)-\widetilde{Z}^{\varepsilon}(s_{i}^{n})\mid ds\\
&&\hspace{.1cm} \leq 2 \sqrt{N} \left(\int_{t_{i}^{n} \wedge \tau_{N}}^{t\wedge \tau_{N}}\left(\mathbb{E} \left(\|\widetilde{\sigma}(s, \widetilde{Z}^{\varepsilon}(s))\|_{L_{Q}} \hspace{.1cm} \mid \widetilde{Z}^{\varepsilon}(s)- \widetilde{Z}^{\varepsilon}(s_{i}^{n})\mid \right)\right)^{2}ds\right)^{1/2}\\
&&\hspace{.1cm} 2\sqrt{N} \left(\int_{t_{i}^{n} \wedge \tau_{N}}^{t\wedge \tau_{N}} \mathbb{E} \|\widetilde{\sigma}(s,\widetilde{Z}^{\varepsilon}(s))\|_{L_{Q}}^{2} \hspace{.1cm} \mathbb{E}\mid \widetilde{Z}^{\varepsilon}(s) - \widetilde{Z}^{\varepsilon}(s_{i}^{n})\mid^{2} ds\right)^{1/2}\\
&&\hspace{.1cm} \leq 2 \sqrt{NK_{9}} \left(\left(1+4\varepsilon \log \log \frac{1}{\varepsilon} \widetilde{C}_{2}(2,\varepsilon) + 2\widetilde{C}_{3}(2)\right)\left(2\widetilde{M}_{1}(T,\varepsilon)\right)\right)^{1/2}\\
&& \hspace{.1cm} \times \mid t\wedge \tau_{N} - t_{i}^{n} \wedge \tau_{N}\mid^{1/2}.
\end{eqnarray*}
Similarly, inequality \eqref{(2.19)} implies,
\begin{eqnarray}\label{(3.19)}
&&\mathbb{E}\mid I_{5}(t_{i}^{n}, t\wedge \tau_{N})\mid\nonumber \\
&&\leq \frac{K_{9}}{2\log \log \frac{1}{\varepsilon}} \left(1+ 4\varepsilon \log \log \frac{1}{\varepsilon} \widetilde{C}_{2}(2,\varepsilon) + 2\widetilde{C}_{3}(2)\right) \mid t\wedge \tau_{N} - t_{i}^{n}\wedge \tau_{N}\mid.
\end{eqnarray}
Since $\mid t\wedge \tau_{N} - t_{i}^{n}\wedge \tau_{N}\mid \leq T 2^{-n}$, then there exists $n_{0}\in \mathbb{N}$, such that for all $n\geq n_{0}$ and $j=0,1,...,5$ except $j=3$,
\begin{equation*}
P\left(I_{j}(t_{i}^{n}\wedge \tau_{N}, t\wedge \tau_{N})>\frac{\beta^{2}}{6}\right) \leq \exp\left(-2R \log \log \frac{1}{\varepsilon}\right),
\end{equation*}
by Chebyshev's inequality, for any given $R>0$ and $0<\varepsilon<\varepsilon_{2}$. As for $j=3$,
\begin{eqnarray*}
&&P\left(I_{3}(t_{i}^{n}\wedge \tau_{N}, t\wedge \tau_{N})> \frac{\beta^{2}}{6}\right)\\
&&= P\left(\int_{t_{i}^{n}\wedge \tau_{N}}^{t\wedge \tau_{N}} \left(\widetilde{\sigma}(s, \widetilde{Z}^{\varepsilon}(s))dW(s), \widetilde{Z}^{\varepsilon}(s)-\widetilde{Z}^{\varepsilon}(s_{i}^{n})\right)>\frac{\beta^{2}}{12} \sqrt{2\log \log \frac{1}{\varepsilon}}\right).
\end{eqnarray*}
Inspired by the proof of Theorem 3.2 in [12] we write,
\begin{eqnarray}\label{(3.20)}
&&\mathbb{E} \exp\left(\mid \int_{t_{i}^{n}\wedge \tau_{N}}^{t\wedge \tau_{N}} \left(\widetilde{\sigma}(s,\widetilde{Z}^{\varepsilon}(s))dW(s), \widetilde{Z}^{\varepsilon}(s)-\widetilde{Z}^{\varepsilon}(s_{i}^{n})\right)\mid^{2}\right)\\
&&= \mathbb{E} \lim_{m\rightarrow \infty} \sum_{k=0}^{m} \frac{1}{k!} \left(\int_{t_{i}^{n}\wedge \tau_{N}}^{t\wedge \tau_{N}} \left(\widetilde{\sigma}(s,\widetilde{Z}^{\varepsilon}(s))dW(s), \widetilde{Z}^{\varepsilon}(s)-\widetilde{Z}^{\varepsilon}(s_{i}^{n})\right)\right)^{2k}, \nonumber
\end{eqnarray}
and since by Burkholder-Davis-Gundy inequality,
\begin{eqnarray*}
&&\mathbb{E}\left(\int_{t_{i}^{n}\wedge \tau_{N}}^{t\wedge \tau_{N}}\left(\widetilde{\sigma}(s,\widetilde{Z}^{\varepsilon}(s))dW(s), \widetilde{Z}^{\varepsilon}(s)-\widetilde{Z}^{\varepsilon}(s_{i}^{n})\right)\right)^{2k}\\
&\leq& \mathbb{E}\left( \int_{t_{i}^{n}\wedge \tau_{N}}^{t\wedge \tau_{N}}\|\widetilde{\sigma}(s, \widetilde{Z}^{\varepsilon}(s))\|_{L_{Q}}^{2} \hspace{.1cm} \mid \widetilde{Z}^{\varepsilon}(s) - \widetilde{Z}^{\varepsilon}(s_{i}^{n})\mid^{2}ds\right)^{k}\\
&\leq& \frac{1}{2} \mathbb{E} \sup_{t_{i}^{n} \wedge \tau_{N} \leq s\leq t\wedge \tau_{N}}  \mid \widetilde{Z}^{\varepsilon}(s)- \widetilde{Z}^{\varepsilon}(s_{i}^{n})\mid^{4k}\\
&& \hspace{.1cm} + \frac{K_{9}}{2}\mathbb{E} \left(\int_{t_{i}^{n}\wedge \tau_{N}}^{t\wedge \tau_{N}} 1 + 4\varepsilon \log \log \frac{1}{\varepsilon} \|\widetilde{Z}^{\varepsilon}(s)\|^{2} + 2\|u^{0}(s)\|^{2}ds\right)^{2k}\\
&\leq& \widetilde{M}_{2}(T,\varepsilon)^{k} + \frac{K_{9}}{2}\left(1 + 4\varepsilon \log \log \frac{1}{\varepsilon} \widetilde{C}_{2}(4k, \varepsilon) + 2\widetilde{C}_{3}(4k)\right) \mid t\wedge \tau_{N}- t_{i}^{n}\wedge \tau_{N}\mid^{2k},
\end{eqnarray*}
we have by Chebyshev's inequality along with an application of Monotone convergence theorem on \eqref{(3.20)},
\begin{equation*}
P\left(I_{3}(t_{i}^{n}\wedge \tau_{N}, t\wedge \tau_{N})>\frac{\beta^{2}}{6}\right) \leq \exp\left(-2R\log \log \frac{1}{\varepsilon}\right),
\end{equation*}
for sufficiently large $n\in \mathbb{N}$, and noting that the above estimate holds for any $\beta>0$, inequality \eqref{(3.18)} is obtained.
\end{proof}
Next we aim to derive the required exponential bound for the second inequality given by \eqref{(3.7)}.
\begin{lemma}\label{Lemm7}
 For any $R>0$ and $\rho>0$, there exists $\eta>0, \beta>0,$ and $n_{0}\in \mathbb{N}\backslash \{0\}$ such that for all $\varepsilon\in(0,\varepsilon_{0})$ and $n\geq n_{0}$,
\begin{flalign}\label{(3.21)}
&P\left(\|\widetilde{Z}^{\varepsilon}(t)-X^{h}(t)\|_{\mathcal{E}(T)}\geq \rho, \left\|\frac{1}{\sqrt{2\log \log \frac{1}{\varepsilon}}} W\right\|_{L_{H}^{\infty}}<\eta, \right. \nonumber&\\
&\hspace{2.5cm}\left. \|\widetilde{Z}^{\varepsilon}(t)-\widetilde{Z}^{\varepsilon}(t_{i}^{n})\|_{\mathcal{E}(T)}\leq \beta\right)\leq \exp\left(-2R\log \log \frac{1}{\varepsilon}\right).&
\end{flalign}
\end{lemma}
\begin{proof}
Observe that under condition, $\|\widetilde{Z}^{\varepsilon}(t)-\widetilde{Z}^{\varepsilon}(t_{i}^{n})\|_{\mathcal{E}(T)}\leq \beta$,
\begin{equation*}
\rho \leq \|\widetilde{Z}^{\varepsilon}(t)-X^{h}(t)\|_{\mathcal{E}(T)}\leq \beta + \|X^{h}(t)-\widetilde{Z}^{\varepsilon}(t_{i}^{n})\|_{\mathcal{E}(T)},
\end{equation*}
Noting that Lemma \ref{Lemm6} holds for any $\beta>0$, then for any $\rho>0$, we may choose $\beta>0$ such that $\xi:= \rho-\beta>0$ to obtain,
\begin{equation*}
\|X^{h}(t)-\widetilde{Z}^{\varepsilon}(t_{i}^{n})\|_{\mathcal{E}(T)} \geq \xi.
\end{equation*}
For $t\in \Delta_{i}^{n}$ we have,
\begin{eqnarray*}
&& X^{h}(t)- \widetilde{Z}^{\varepsilon}(t_{i}^{n}) = X^{h}(0) - \int_{0}^{t_{i}^{n}} \left(AX^{h}(s) - A\widetilde{Z}^{\varepsilon}(s)\right)ds - \int_{t_{i}^{n}}^{t} AX^{h}(s)ds\\
&&- \int_{0}^{t_{i}^{n}} B\left(X^{h}(s)- \widetilde{Z}^{\varepsilon}(s), u^{0}(s)\right)ds - \int_{t_{i}^{n}}^{t} B(X^{h}(s), u^{0}(s))ds\\
&&- \int_{0}^{t} B(u^{0}(s), X^{h}(s))ds + \int_{0}^{t_{i}^{n}} B(u^{\varepsilon}(s), \widetilde{Z}^{\varepsilon}(s))ds + \int_{0}^{t} \sigma(s,u^{0}(s))h(s)ds\\
&& - \int_{0}^{t_{i}^{n}} \widetilde{\sigma}(s,\widetilde{Z}^{\varepsilon}(s))h(s)ds - \frac{1}{\sqrt{2\log \log \frac{1}{\varepsilon}}} \int_{0}^{t_{i}^{n}}\widetilde{\sigma}(s,\widetilde{Z}^{\varepsilon}(s))dW(s).
\end{eqnarray*}
Applying similar estimates as in the proof of Lemma \ref{Lemm6} with bound in \eqref{(3.17)} we arrive at,
\begin{eqnarray*}
&&\mathbb{E}\sup_{0\leq s\leq t} \mid X^{h}(s) - \widetilde{Z}^{\varepsilon}(s_{i}^{n})\mid^{2} + \mathbb{E}\int_{0}^{T} \|X^{h}(s)- \widetilde{Z}^{\varepsilon}(s_{i}^{n})\|^{2}ds\\
&\leq& \mid X^{h}(0)\mid^{2} + c_{1}t_{i}^{n} + c_{2}\mid t-t_{i}^{n}\mid + c_{3}t\\
&&+ \frac{2}{\sqrt{2\log \log \frac{1}{\varepsilon}}} \mathbb{E}\mid \int_{0}^{t_{i}^{n}} \left(\widetilde{\sigma}(s,\widetilde{Z}^{\varepsilon}(s))dW, X^{h}(s)-\widetilde{Z}^{\varepsilon}(s_{i}^{n})\right)\mid,
\end{eqnarray*}
where $c_{1}, c_{2}, c_{3}$ are constants that depend on $\varepsilon, T, \widetilde{C}_{2}(2,\varepsilon), \widetilde{C}_{3}(2)$ and $\widetilde{C}_{4}(2)$. Thus, noting that $t\in \Delta_{i}^{n}$, and following the same reasoning as in the proof of Lemma \ref{Lemm6} we achieve \eqref{(3.21)}.
\end{proof}
It may be observed that the above results can be generalized to achieve the moderate deviations for $\{u^{\varepsilon}(\cdot)\}_{\varepsilon>0}$ by the Azencott method with any $a(\varepsilon)$ satisfying the required conditions, $a(\varepsilon)>0$ and $a(\varepsilon)/\sqrt{\varepsilon} \rightarrow \infty$ as $\varepsilon$ tends to zero. Here we focused on the moderate deviation principle for the special case of $a(\varepsilon)= 1/\sqrt{2\log \log(1/\varepsilon)}$ to be able to achieve the LIL in the next section.
\section{S\lowercase{trassen's} C\lowercase{ompact} LIL}
We begin by showing the relative compactness property of the process $\{Z^{\varepsilon}(\cdot)\}_{\varepsilon\in(0,\varepsilon_{0})}$ in space $\mathcal{C}([0,T];H)\cap L^{2}(0,T;V)$ as required by Theorem \ref{Them2}. To this end, the following result proved in [26] is applied, where the statement of the lemma is modified to match our setting. We make the remark that since the global in time well-posedness of solutions is known for our process, the convergence offered by the theorem holds for any time $t\in [0,T]$ and the use of stopping times is not necessary.
\begin{theorem}(Lemma 5.1 in [26]).
Let $B_{1}$ and $B_{2}$ be Banach spaces with norms, $\|\cdot\|_{1}$ and $\|\cdot\|_{2}$, respectively such that $B_{2}\subset B_{1}$ is a continuous embedding. Suppose $\{X^{\varepsilon}\}_{\varepsilon>0}$ is a family of $B_{2}$-valued stochastic process defined on $\mathcal{E}(T):= \mathcal{C}([0,T];B_{1})\cap L^{2}(0,T;B_{2})$ a.s. If for some $M>1$ and $T>0$,
\begin{eqnarray}
&&\lim_{\varepsilon_{1}\rightarrow 0} \sup_{\varepsilon_{1}\geq \varepsilon_{2}} \mathbb{E} \|X^{\varepsilon_{1}} - X^{\varepsilon_{2}}\|_{\mathcal{E}(T)} =0, \label{(4.1)}\\
&& \lim_{S\rightarrow 0} \sup_{\varepsilon>0} P\left(\|X^{\varepsilon}\|_{\mathcal{E}(T\wedge S)} > \|X^{\varepsilon}(0)\|_{1} + M-1\right) =0, \label{(4.2)}
\end{eqnarray}
then for some subsequence, $\{X^{\varepsilon_{\ell}}\}_{\varepsilon_{\ell}>0}$ and process $X\in \mathcal{E}(T)$, the convergence, $\|X^{\varepsilon_{\ell}}-X\|_{\mathcal{E}(T)}\rightarrow 0$ holds a.s. as $\varepsilon_{\ell} \rightarrow 0$.
\end{theorem}
We proceed by verifying conditions \eqref{(4.1)} and \eqref{(4.2)} for our model. Observing that $Z^{\varepsilon}(t)$ has the same terms as $\widetilde{Z}^{\varepsilon}(t)$ with an addition of the term, $\int_{0}^{t}\widetilde{\sigma}(s,\widetilde{Z}^{\varepsilon}(s))h(s)ds$, we deduce estimates \eqref{(3.8)}-\eqref{(3.10)} and \eqref{(3.15)} for $Z^{\varepsilon}(t)$ and for simplicity keep the same notation for the bounds. \\
\hspace{.5cm} Let $V^{\varepsilon}(t)= Z^{\varepsilon_{1}}(t)-Z^{\varepsilon_{2}}(t)$, then applying the It$\hat{o}$ formula, taking the supremum over time up to $\tau_{M}:= \inf\{t: \|V^{\varepsilon}(t)\|_{\mathcal{E}(T)} \geq M\}$ for some $M>0$, and afterwards taking the expectation gives,
\begin{flalign*}
&\mathbb{E}\sup_{0\leq s\leq t\wedge \tau_{M}} \mid V^{\varepsilon}(s)\mid^{2} + 2\int_{0}^{t\wedge \tau_{M}} \mathbb{E} \|V^{\varepsilon}(s)\|^{2}ds&\\
&\hspace{.1cm}\leq 2\mathbb{E} \int_{0}^{t\wedge \tau_{M}} \left(-\left(B\left(Z^{\varepsilon_{1}}(s), u^{\varepsilon_{1}}(s)\right), V^{\varepsilon}(s)\right) + \left(B\left(Z^{\varepsilon_{2}}(s), u^{\varepsilon_{2}}(s)\right), V^{\varepsilon}(s)\right)\right)ds&\\
&\hspace{.1cm} + 2\mathbb{E} \sup_{0\leq s\leq t\wedge \tau_{M}}\int_{0}^{s}\left(\left(\frac{\widetilde{\sigma}(\ell, Z^{\varepsilon_{1}}(\ell))}{\sqrt{2\log \log \frac{1}{\varepsilon_{1}}}} - \frac{\widetilde{\sigma}(\ell, Z^{\varepsilon_{2}}(\ell))}{\sqrt{2\log \log \frac{1}{\varepsilon_{2}}}}\right) dW(\ell), V^{\varepsilon}(s)\right)&\\
&\hspace{.1cm} \mathbb{E}\int_{0}^{t\wedge \tau_{M}}\left( \frac{1}{2\log \log \frac{1}{\varepsilon_{1}}} \|\widetilde{\sigma}(s, Z^{\varepsilon_{1}}(s))\|_{L_{Q}}^{2} + \frac{1}{2\log \log \frac{1}{\varepsilon_{2}}} \|\widetilde{\sigma}(s,Z^{\varepsilon_{2}}(s))\|_{L_{Q}}^{2} \right)ds,&
\end{flalign*}
where, for the first term on the right hand side, estimates may be made along the same lines as those in the proof of Lemma \ref{Lemm3} to obtain,
\begin{flalign*}
&\lim_{\varepsilon_{1}\rightarrow 0} \sup_{\varepsilon_{1}\geq \varepsilon_{2}} \mathbb{E}\left(\sup_{s\in [0,t\wedge \tau_{M}]} \mid Z^{\varepsilon_{1}}(s)-Z^{\varepsilon_{2}}(s)\mid^{2} + \mathbb{E} \int_{0}^{t\wedge \tau_{N}} \|Z^{\varepsilon_{1}}(s)- Z^{\varepsilon_{2}}(s)\|^{2} ds\right)&\\
&=0,&
\end{flalign*}
which is equivalent to \eqref{(4.1)} after setting $M$ tend to infinity. As for \eqref{(4.2)}, applying It$\hat{o}$'s formula then taking the supremum over $t\in [0,T\wedge S]$ yields,
\begin{eqnarray*}
&&\sup_{t\in [0,T\wedge S]} \mid Z^{\varepsilon}(t)\mid^{2} + \int_{0}^{T\wedge S} \|Z^{\varepsilon}(s)\|^{2}ds \leq 2c \int_{0}^{T\wedge S} \mid Z^{\varepsilon}(s)\mid^{2}\|u^{\varepsilon}(s)\|_{L^{4}}^{4}ds\\
&&+ \frac{2}{\sqrt{2\log \log \frac{1}{\varepsilon}}} \sup_{s\in [0,T\wedge S]} \int_{0}^{s}\left(\widetilde{\sigma}(\ell, Z^{\varepsilon}(\ell))dW(\ell), Z^{\varepsilon}(\ell)\right)\\
&& + \frac{1}{2\log \log \frac{1}{\varepsilon}} \int_{0}^{T\wedge S} \|\widetilde{\sigma}(s, Z^{\varepsilon}(s))\|_{L_{Q}}^{2}ds,
\end{eqnarray*}
where inequality \eqref{(2.4)} was applied. Hence, we obtain,
\begin{eqnarray*}
&&P\left(\|Z^{\varepsilon}(t)\|^{2}_{\mathcal{E}(T\wedge S)} >M-1\right)\\
&& \leq P\left(2c\int_{0}^{T\wedge S} \mid Z^{\varepsilon}(s)\mid^{2} \|u^{\varepsilon}(s)\|_{L^{4}}^{4} ds>\frac{(M-1)}{3}\right)\\
&& + P\left(\frac{2}{\sqrt{2\log \log \frac{1}{\varepsilon}}} \sup_{s\in [0,T\wedge S]}\int_{0}^{s} \left(\widetilde{\sigma}(\ell, Z^{\varepsilon}(\ell))dW(\ell), Z^{\varepsilon}(\ell)\right)>\frac{(M-1)}{3}\right)\\
&&+ P\left(\frac{1}{2\log \log \frac{1}{\varepsilon}} \int_{0}^{T\wedge S} \|\widetilde{\sigma}(s, Z^{\varepsilon}(s))\|_{L_{Q}}^{2} ds>\frac{(M-1)}{3}\right).
\end{eqnarray*}
We may apply Doob's and Chebyshev inequalities for the first and remaining two probabilities, respectively, to arrive at,
\begin{eqnarray*}
&&P\left(\sup_{t\in[0,T\wedge S]} \mid Z^{\varepsilon}(t)\mid^{2} + \int_{0}^{T\wedge S} \|Z^{\varepsilon}(s)\|^{2} ds>M-1\right)\\
&& \leq \frac{6c}{(M-1)} \mathbb{E}\left(\int_{0}^{T\wedge S} \mid Z^{\varepsilon}(s)\mid^{2} \hspace{.1cm} \|u^{\varepsilon}(s)\|_{L^{4}}^{4} ds\right)\\
&&+ \frac{6K_{9}}{(M-1)\log \log \frac{1}{\varepsilon}} \mathbb{E} \int_{0}^{T\wedge S}(1+ 4\varepsilon \log \log \frac{1}{\varepsilon} \|Z^{\varepsilon}(s)\|^{2} + 2\|u^{0}(s)\|^{2}) \mid Z^{\varepsilon}(s)\mid^{2}ds\\
&& + \frac{3K_{9}}{2(M-1)\log \log \frac{1}{\varepsilon}} \mathbb{E}\int_{0}^{T\wedge S}(1+4\varepsilon \log \log \frac{1}{\varepsilon} \|Z^{\varepsilon}(s)\|^{2} + 2\|u^{0}(s)\|^{2}) ds\\
&&\leq K(\varepsilon, T)(T\wedge S),
\end{eqnarray*}
for a positive constant $K(\varepsilon,T)$, where the last inequality was achieved by applying H\"older's inequality similar to estimates in the proof of Lemma \ref{Lemm6}. Thus, by taking the supremum on $\varepsilon\in(0,\varepsilon_{0})$ and afterwards letting $S$ tend to zero we achieve condition \eqref{(4.2)} in our setting. \\
\hspace{.5cm} Next to verify that the set $L$ given in Theorem \ref{Them2} is the limit set, we prove by the following lemma that each element of set $L$ is a limit point of $\{Z^{\varepsilon}(\cdot)\}_{\varepsilon \in (0,\varepsilon_{0})}$. For this lemma, we let $c>1$ and consider the process depending on $1/c^{j}$ for $j\geq1$, instead of $\varepsilon>0$ for better presentation.
\begin{lemma}
For any $c>1$ and $g(t)\in L$, there exists $j_{0}>\frac{1}{\log c} \log \frac{1}{\varepsilon_{0}}$, such that,
\begin{equation*}
P\left(\|Z^{\frac{1}{c^{j}}}(t)-g(t)\|_{\mathcal{E}(T)} \leq \varepsilon \hspace{.2cm} i.o.\right)=1,
\end{equation*}
for all $j\geq j_{0}$ and $\varepsilon>0$.
\end{lemma}
\begin{proof}
For a constant $\eta>0$, let,
\begin{equation*}
F_{j}:=\left\{\|Z^{\frac{1}{c^{j}}}(t)-g(t)\|_{\mathcal{E}(T)} \leq \varepsilon\right\}, \hspace{.5cm} G_{j}:= \left\{\left\|\frac{1}{\sqrt{2\log \log c^{j}}} W-h\right\|_{L_{H}^{\infty}} \leq \eta \right\},
\end{equation*}
where $g(t)$ is any element in the set $L$ and $h\in \mathcal{H}_{0}$ such that $g(t)=X^{h}(t)$ and \\
$\frac{1}{2}\int_{0}^{T} \mid h(s)\mid_{0}^{2}ds\leq \frac{N}{2}$. Since the Strassen's compact LIL is known for Brownian paths (see [36]), we have,
\begin{equation}\label{(4.3)}
P\left(\limsup_{j\rightarrow \infty} \left\|\frac{1}{\sqrt{2\log \log c^{j}}}W -h\right\|_{L_{H_{0}}^{\infty}} > \eta\right)=0,
\end{equation}
Also note that by the Freidlin-Wentzell inequality \eqref{(3.4)} we have for all $R>1$,
\begin{equation*}
P(F_{j}^{c} \cap G_{j})\leq \exp(-2R\log \log c^{j})\leq \frac{C}{j^{2R}},
\end{equation*}
which by the Borel-Cantelli lemma yields,
\begin{equation}\label{(4.4)}
P\left(\limsup_{j\rightarrow \infty} F_{j}^{c} \cap G_{j}\right)=0.
\end{equation}
Now by \eqref{(4.3)},
\begin{eqnarray*}
1= P\left(\limsup_{j\rightarrow \infty} G_{j}\right) &\leq& P\left(\limsup_{j\rightarrow \infty} G_{j}\cap F_{j}\right) + P\left(\limsup_{j\rightarrow \infty} G_{j}\cap F_{j}^{c}\right)\\
&\leq& P\left(\limsup_{j\rightarrow \infty} F_{j}\right),
\end{eqnarray*}
 giving $P(\limsup_{j\rightarrow \infty} F_{j})=1$ to complete the proof.
 \end{proof}
 Inspired by the proof of Proposition 3.2 of [11], to ensure that set $L$ is the only limit set of $\{Z^{\varepsilon}(\cdot)\}_{\varepsilon \in (0,\varepsilon_{0})}$, we let $\overline{L}:= \{g: \|g(t)-L\|_{\mathcal{E}(T)}\geq \varepsilon \}$, which implies by the definition of set $L$ that $I(g)>N/2 + \delta$ for some $\delta>0$. By the MDP result of the previous section, we have for the closed set, $\overline{L}$,
 \begin{equation*}
 \limsup_{j\rightarrow \infty} \frac{1}{2\log \log c^{j}} \log P\left(Z^{\frac{1}{c^{j}}}\in \overline{L}\right) \leq -I(g)< - \left(\frac{N}{2}+\delta\right),
 \end{equation*}
 which leads to,
 \begin{equation*}
 P(Z^{\frac{1}{c^{j}}} \in \overline{L}) \leq \exp\left(-2\left(\frac{N}{2}+\delta\right)\log \log c^{j}\right) \leq \frac{k}{j^{N+2\delta}},
 \end{equation*}
 giving,
 \begin{equation*}
 P\left(\limsup_{j\rightarrow \infty}\|Z^{\frac{1}{c^{j}}}(t)-\overline{L}\|_{\mathcal{E}(T)} \geq \varepsilon\right)=0,
 \end{equation*}
 by Borel-Cantelli lemma, where we have noted the assumption of $N$ being a natural number. Thus, we achieve the limit of $\|Z^{\frac{1}{c^{j}}}(t)-L\|_{\mathcal{E}(T)}$ to zero a.s.
\section*{D\lowercase{eclarations}}
The author declares that there is no conflict of interest.
\section*{A\lowercase{ppendix}}\label{secA1}
\hspace{.5cm} \emph{Proof of Lemma \ref{Lemm3}}\\
\hspace{.5cm} Letting $\tau_{N}:= \inf\{t>0:\sup_{0\leq t\leq T} \mid \widetilde{Z}^{\varepsilon}(t)\mid ^{2} + \int_{0}^{t}\|\widetilde{Z}^{\varepsilon}(s)\|^{2}ds >N\}$ we apply the It$\hat{o}$'s formula, then take the supremum over time up to $t\wedge \tau_{N}$ and afterwards expectation to obtain,
\begin{eqnarray*}
&&\mathbb{E} \sup_{0\leq s\leq t\wedge \tau_{N}} \mid \widetilde{Z}^{\varepsilon}(s)\mid^{2} + 2\mathbb{E}\int_{0}^{t\wedge \tau_{N}} \|\widetilde{Z}^{\varepsilon}(s)\|^{2}ds\\
&=& -2\mathbb{E}\int_{0}^{t\wedge \tau_{N}} \left(B\left(\widetilde{Z}^{\varepsilon}(s), u^{0}(s)\right), \widetilde{Z}^{\varepsilon}(s)\right) ds \\
&& + 2\mathbb{E}\int_{0}^{t\wedge \tau_{N}} \left(\widetilde{\sigma}\left(s,\widetilde{Z}^{\varepsilon}(s)\right)h(s), \widetilde{Z}^{\varepsilon}(s)\right)ds\\
&&+ \frac{2}{\sqrt{2\log \log \frac{1}{\varepsilon}}} \mathbb{E} \sup_{0\leq s\leq t\wedge \tau_{N}} \int_{0}^{s} \left(\widetilde{\sigma}(\ell, \widetilde{Z}^{\varepsilon}(\ell))dW(\ell), \widetilde{Z}^{\varepsilon}(\ell)\right)\\
&&+ \frac{1}{2\log \log \frac{1}{\varepsilon}} \mathbb{E} \int_{0}^{t\wedge \tau_{N}} \|\widetilde{\sigma}(s,\widetilde{Z}^{\varepsilon}(s))\|_{L_{Q}}^{2}ds\\
&=& I_{1}(t\wedge \tau_{N}) + I_{2}(t\wedge \tau_{N}) + \mathbb{E}\sup_{0\leq s\leq t\wedge \tau_{N}}I_{3}(s)+I_{4}(t\wedge \tau_{N}).
\end{eqnarray*}
We proceed to estimate the above terms. Using inequality \eqref{(2.4)} we have,
\begin{flalign}\label{(5.1)}
&I_{1}(t\wedge \tau_{N}) = 2\mathbb{E} \int_{0}^{t\wedge \tau_{N}} b(\widetilde{Z}^{\varepsilon}(s), \widetilde{Z}^{\varepsilon}(s), u^{0}(s))ds&\nonumber \\
&\leq \int_{0}^{t\wedge \tau_{N}} \mathbb{E}\|\widetilde{Z}^{\varepsilon}(s)\|^{2}ds + 2c\int_{0}^{t\wedge \tau_{N}} \mathbb{E} \sup_{0\leq \ell\leq s} \mid \widetilde{Z}^{\varepsilon}(\ell)\mid^{2} \|u^{0}(s)\|^{4}_{L^{4}}ds.&
\end{flalign}
By Cauchy-Schwarz and Young's inequalities,
\begin{eqnarray*}
I_{2}(t\wedge \tau_{N}) &\leq& \frac{1}{2}\mathbb{E} \int_{0}^{t\wedge \tau_{N}} \|\widetilde{\sigma}(s,\widetilde{Z}^{\varepsilon}(s))\|_{L_{Q}}^{2} ds + 2\mathbb{E} \int_{0}^{t\wedge \tau_{N}}\mid h(s)\mid_{0}^{2}\hspace{.1cm}\mid\widetilde{Z}^{\varepsilon}(s)\mid^{2}ds\\
&\leq& \frac{K_{9}}{2}(T+2K_{6}) + 2\varepsilon K_{9} \left(\log \log \frac{1}{\varepsilon}\right) \mathbb{E}\int_{0}^{t\wedge \tau_{N}} \|\widetilde{Z}^{\varepsilon}(s)\|^{2}ds\\
&& + 2\mathbb{E}\int_{0}^{t\wedge \tau_{N}} \sup_{0\leq \ell\leq s}\mid \widetilde{Z}^{\varepsilon}(\ell)\mid^{2}\hspace{.1cm} \mid h(s)\mid_{0}^{2}ds,
\end{eqnarray*}
where inequality \eqref{(2.19)} was also applied. Thanks to the Burkholder-Davis-Gundy inequality and inequality \eqref{(2.11)} we obtain,
\begin{flalign*}
& \sup_{0\leq s\leq t\wedge \tau_{N}} I_{3}(s)\leq \frac{6}{\sqrt{2\log \log \frac{1}{\varepsilon}}} \mathbb{E}\left(\int_{0}^{t\wedge \tau_{N}} \|\widetilde{\sigma}(s,\widetilde{Z}^{\varepsilon}(s))\|_{L_{Q}}^{2} \hspace{.1cm} \mid \widetilde{Z}^{\varepsilon}(s)\mid^{2}ds\right)^{1/2}&\\
& \leq \frac{6\sqrt{K_{9}}}{\sqrt{2\log \log \frac{1}{\varepsilon}}} \mathbb{E}\left(\int_{0}^{t\wedge \tau_{N}} (1+ 4\varepsilon \log \log \frac{1}{\varepsilon}\|\widetilde{Z}^{\varepsilon}(s)\|^{2} + 2\|u^{0}(s)\|^{2}) \mid \widetilde{Z}^{\varepsilon}(s)\mid^{2}ds\right)^{1/2}&\\
&\leq \frac{1}{2}\mathbb{E} \sup_{0\leq s\leq t\wedge \tau_{N}} \mid \widetilde{Z}^{\varepsilon}(s)\mid^{2}&\\
& + \frac{9K_{9}}{\log \log \frac{1}{\varepsilon}}\int_{0}^{t\wedge \tau_{N}} (1+ 4\varepsilon \log \log \frac{1}{\varepsilon}\mathbb{E} \|\widetilde{Z}^{\varepsilon}(s)\|^{2} + 2\|u^{0}(s)\|^{2})ds&\\
&\leq \frac{1}{2} \mathbb{E} \sup_{0\leq s\leq t\wedge \tau_{N}} \mid \widetilde{Z}^{\varepsilon}(s)\mid^{2} + \frac{9K_{9}}{\log \log \frac{1}{\varepsilon}}(T+2K_{6}) + 36\varepsilon K_{9} \int_{0}^{t\wedge \tau_{N}} \mathbb{E} \|\widetilde{Z}^{\varepsilon}(s)\|^{2}ds,&
\end{flalign*}
In addition,
\begin{equation*}
I_{4}(t\wedge \tau_{N}) \leq \frac{K_{9}}{2\log \log \frac{1}{\varepsilon}} (T+2K_{6}) + 2\varepsilon K_{9}\mathbb{E} \int_{0}^{t\wedge \tau_{N}} \|\widetilde{Z}^{\varepsilon}(s)\|^{2}ds.
\end{equation*}
Hence, we arrive at,
\begin{flalign*}
&\frac{1}{2} \mathbb{E}\sup_{0\leq s\leq t\wedge \tau_{N}} \mid \widetilde{Z}^{\varepsilon}(s)\mid ^{2} + \left(1-36\varepsilon K_{9} -2\varepsilon K_{9} \log \log \frac{1}{\varepsilon}\right) \int_{0}^{t\wedge \tau_{N}} \mathbb{E} \|\widetilde{Z}^{\varepsilon}(s)\|^{2}ds&\\
&\leq M_{1}(\varepsilon, T) + \int_{0}^{t\wedge \tau_{N}}\mathbb{E}\sup_{0\leq \ell\leq s} \mid \widetilde{Z}^{\varepsilon}(\ell)\mid^{2} \left(2c\|u^{0}(s)\|^{4}_{L^{4}} + 2\mid h(s)\mid_{0}^{2}\right)ds,&
\end{flalign*}
where,
\begin{equation*}
M_{1}(\varepsilon, T) := \frac{K_{9}}{\log \log \frac{1}{\varepsilon}} \left(9T + 18K_{6} + \frac{T}{2} + K_{6}\right) + \frac{K_{9}}{2} \left(T+ 2K_{6}\right).
\end{equation*}
Observe that to ensure the condition, $1-36\varepsilon K_{9}-2\varepsilon K_{9}\log \log 1/\varepsilon >0$, we need $\varepsilon<1$ and $\varepsilon < 1/(36K_{9})$. Now an application of Gronwall's inequality yields,
\begin{eqnarray*}
&& \mathbb{E} \sup_{0\leq s\leq t\wedge \tau_{N}} \mid \widetilde{Z}^{\varepsilon}(s)\mid^{2} + \int_{0}^{t\wedge \tau_{N}} \mathbb{E} \|\widetilde{Z}^{\varepsilon}(s)\|^{2}ds\\
&\leq& K M_{1}(\varepsilon, T)\exp\left(\int_{0}^{t\wedge \tau_{N}} \left(2c\|u^{0}(s)\|^{4}_{L^{4}} + 2\mid h(s)\mid_{0}^{2}\right)ds\right) =: \widetilde{M}_{1}(\varepsilon, T).
\end{eqnarray*}
The result follows by noting that $h\in \mathcal{H}_{0}$, using inequality \eqref{(2.12)} and letting $N$ tend to infinity. \\
\hspace{.5cm} We prove the following lemma in which the first inequality is similar to the proof of Lemma 3.2 in [38]; however, we derive it here for our setting for completeness. \\
\hspace{.5cm} \emph{Proof of Lemma \ref{Lemm4}}\\
\hspace{.5cm} To achieve \eqref{(3.9)}, we let, \\
$\tau_{N}:= \inf \{t>0: \sup_{0\leq s\leq t} \mid \widetilde{Z}^{\varepsilon}(s)\mid^{4} + \int_{0}^{t}\mid\widetilde{Z}^{\varepsilon}(s)\mid^{2} \|\widetilde{Z}^{\varepsilon}(s)\|^{2}ds>N\}$ and applying the It$\hat{o}$'s formula first to $\mid \widetilde{Z}^{\varepsilon}(s)\mid^{2}$ then to the map $x\mapsto x^{2}$ we obtain,
\begin{equation*}
d(\mid \widetilde{Z}^{\varepsilon}(t)\mid ^{2})^{2} = 2\mid \widetilde{Z}^{\varepsilon}(t)\mid^{2} d\mid \widetilde{Z}^{\varepsilon}(t)\mid^{2} + d< \mid\widetilde{Z}^{\varepsilon}(s)\mid^{2}>_{t}.
\end{equation*}
Namely,
\begin{eqnarray*}
\mid \widetilde{Z}^{\varepsilon}(t\wedge \tau_{N})\mid^{4}&\leq& -4 \int_{0}^{t\wedge \tau_{N}} \mid \widetilde{Z}^{\varepsilon}(s)\mid^{2}\hspace{.1cm} \|\widetilde{Z}^{\varepsilon}(s)\|^{2}ds \\
&&- 4 \int_{0}^{t\wedge \tau_{N}} \mid \widetilde{Z}^{\varepsilon}(s)\mid^{2} \left(B(\widetilde{Z}^{\varepsilon}(s), u^{0}(s)), \widetilde{Z}^{\varepsilon}(s)\right)ds\\
&&+ \frac{4}{\sqrt{2\log \log \frac{1}{\varepsilon}}} \int_{0}^{t\wedge \tau_{N}} \mid \widetilde{Z}^{\varepsilon}(s)\mid^{2} \left(\widetilde{\sigma}(s, \widetilde{Z}^{\varepsilon}(s))dW(s), \widetilde{Z}^{\varepsilon}(s)\right)\\
&&+4\int_{0}^{t\wedge \tau_{N}} \mid \widetilde{Z}^{\varepsilon}(s)\mid^{2} \left(\widetilde{\sigma}(s,\widetilde{Z}^{\varepsilon}(s))h(s), \widetilde{Z}^{\varepsilon}(s)\right)ds\\
&&+ \frac{3}{\log \log \frac{1}{\varepsilon}} \int_{0}^{t\wedge \tau_{N}} \mid \widetilde{Z}^{\varepsilon}(s)\mid^{2} \|\widetilde{\sigma}(s, \widetilde{Z}^{\varepsilon}(s))\|_{L_{Q}}^{2}ds\\
&=& I_{0}(t\wedge \tau_{N}) + I_{1}(t\wedge \tau_{N})+ I_{2}(t\wedge \tau_{N}) + I_{3}(t\wedge \tau_{N}) + I_{4}(t\wedge \tau_{N}).
\end{eqnarray*}
Now we take the supremum over time and then expectation to arrive at,
\begin{eqnarray*}
&& \mathbb{E}\sup_{0\leq s\leq t\wedge \tau_{N}} \mid \widetilde{Z}^{\varepsilon}(s)\mid^{4} + 4 \mathbb{E} \int_{0}^{t\wedge \tau_{N}} \mid \widetilde{Z}^{\varepsilon}(s)\mid^{2}\|\widetilde{Z}^{\varepsilon}(s)\|^{2}ds\\
&&\leq \mathbb{E} I_{1}(t\wedge \tau_{N}) + \mathbb{E}\sup_{0\leq s\leq t\wedge \tau_{N}} I_{2}(s) + \mathbb{E} I_{3}(t\wedge \tau_{N}) + \mathbb{E}I_{4}(t\wedge \tau_{N}).
\end{eqnarray*}
Similar to estimates in the proof of previous lemma we find,
\begin{eqnarray*}
&&\mathbb{E}I_{1}(t\wedge \tau_{N}) = 4\mathbb{E}\int_{0}^{t \wedge \tau_{N}} \mid \widetilde{Z}^{\varepsilon}(s)\mid^{2} B(\widetilde{Z}^{\varepsilon}(s), \widetilde{Z}^{\varepsilon}(s), u^{0}(s))ds\\
&\leq& 2\mathbb{E} \int_{0}^{t\wedge \tau_{N}} \mid \widetilde{Z}^{\varepsilon}(s)\mid^{2} \|\widetilde{Z}^{\varepsilon}(s)\|^{2} ds + 4c \mathbb{E} \int_{0}^{t\wedge \tau_{N}} \sup_{0\leq \ell \leq s} \mid \widetilde{Z}^{\varepsilon}(\ell)\mid^{4} \|u^{0}(s)\|^{4}_{L^{4}}ds.
\end{eqnarray*}
By Burkholder-Davis-Gunday inequality,
\begin{eqnarray*}
&& \mathbb{E}\sup_{0\leq s\leq t} I_{2}(s)\leq \frac{12}{\sqrt{2\log \log \frac{1}{\varepsilon}}} \mathbb{E}\left(\int_{0}^{t\wedge \tau_{N}} \mid \widetilde{Z}^{\varepsilon}(s)\mid^{6} \|\widetilde{\sigma}(s, \widetilde{Z}^{\varepsilon}(s))\|_{L_{Q}}^{2} ds\right)^{1/2}\\
&\leq& \frac{12}{\sqrt{2\log \log \frac{1}{\varepsilon}}} \mathbb{E} \left(\sup_{0\leq s\leq t\wedge \tau_{N}} \mid \widetilde{Z}^{\varepsilon}(s)\mid^{4} \int_{0}^{t\wedge \tau_{N}} \mid \widetilde{Z}^{\varepsilon}(s)\mid^{2}\|\widetilde{\sigma}(s,\widetilde{Z}^{\varepsilon}(s))\|_{L_{Q}}^{2} ds\right)^{1/2}\\
&\leq& \frac{1}{2} \mathbb{E}\sup_{0\leq s\leq t\wedge \tau_{N}} \mid \widetilde{Z}^{\varepsilon}(s)\mid^{4} + 144 \varepsilon K_{9}\mathbb{E} \int_{0}^{t\wedge \tau_{N}} \mid \widetilde{Z}^{\varepsilon}(s)\mid^{2} \|\widetilde{Z}^{\varepsilon}(s)\|^{2}ds\\
&&+ \frac{36K_{9}}{\log \log \frac{1}{\varepsilon}} \left(T+2K_{6}\right)\widetilde{M}_{1}(T,\varepsilon).
\end{eqnarray*}
Inequality \eqref{(2.19)} may be used again to yield,
\begin{eqnarray*}
&&\mathbb{E}I_{3}(t\wedge \tau_{N}) \leq 4 \mathbb{E}\int_{0}^{t\wedge \tau_{N}} \mid \widetilde{Z}^{\varepsilon}(s)\mid^{2} \|\widetilde{\sigma}(s, \widetilde{Z}^{\varepsilon}(s))\|_{L_{Q}} \mid h(s)\mid_{0} \mid \widetilde{Z}^{\varepsilon}(s)\mid ds\\
&\leq& 8 \mathbb{E} \int_{0}^{t\wedge \tau_{N}} \sup_{0\leq \ell \leq s} \mid \widetilde{Z}^{\varepsilon}(s)\mid^{4}\mid h(s)\mid_{0}ds\\
&& + \frac{K_{9}}{2}\mathbb{E}\int_{0}^{t\wedge \tau_{N}} \left(1+ 4\varepsilon \log \log \frac{1}{\varepsilon}\|\widetilde{Z}^{\varepsilon}(s)\|^{2} + 2\|u^{0}(s)\|^{2}\right) \mid\widetilde{Z}^{\varepsilon}(s)\mid^{2}ds\\
&\leq& 8 \mathbb{E} \int_{0}^{t\wedge \tau_{N}} \sup_{0\leq \ell \leq s}\mid \widetilde{Z}^{\varepsilon}(s)\mid^{4} \mid h(s)\mid_{0}^{2}ds + K_{9} \widetilde{M}_{1}(T,\varepsilon)\left(\frac{T}{2} + K_{6}\right)\\
&&+ 2 \varepsilon K_{9}\log \log \frac{1}{\varepsilon} \mathbb{E}\int_{0}^{t\wedge \tau_{N}} \|\widetilde{Z}^{\varepsilon}(s)\|^{2} \mid \widetilde{Z}^{\varepsilon}(s)\mid^{2}ds,
\end{eqnarray*}
and
\begin{eqnarray*}
\mathbb{E}I_{4}(t\wedge \tau_{N}) &\leq& \frac{3K_{9}}{\log \log \frac{1}{\varepsilon}} \widetilde{M}_{1}(T,\varepsilon) \left(T+ 2K_{6}\right)\\
&& + 12 \varepsilon K_{9}\mathbb{E} \int_{0}^{t\wedge \tau_{N}} \mid \widetilde{Z}^{\varepsilon}(s)\mid^{2}\|\widetilde{Z}^{\varepsilon}(s)\|^{2}ds.
\end{eqnarray*}
Thus, we have,
\begin{eqnarray*}
&& \frac{1}{2} \mathbb{E}\sup_{0\leq s\leq t\wedge \tau_{N}} \mid \widetilde{Z}^{\varepsilon}(s)\mid^{4}\\
&& + \left(2-156 \varepsilon K_{9} -2\varepsilon K_{9}\log \log \frac{1}{\varepsilon}\right) \mathbb{E}\int_{0}^{t\wedge \tau_{N}}\mid \widetilde{Z}^{\varepsilon}(s)\mid^{2} \|\widetilde{Z}^{\varepsilon}(s)\|^{2}ds\\
&\leq& M_{2}(\varepsilon, T) + \mathbb{E}\int_{0}^{t\wedge \tau_{N}} \sup_{0\leq \ell \leq s} \mid \widetilde{Z}^{\varepsilon}(s)\mid^{4} \left(4c\|u^{0}(s)\|_{L^{4}}^{4} + 8\mid h(s)\mid_{0}^{2}\right)ds,
\end{eqnarray*}
where,
\begin{equation*}
M_{2}(\varepsilon, T):= \frac{39 K_{9}\widetilde{M}_{1}(T, \varepsilon)}{\log \log \frac{1}{\varepsilon}} (T+ 2K_{6}) + K_{9}\widetilde{M}_{1}(T,\varepsilon) \left(\frac{T}{2} + K_{6}\right).
\end{equation*}
Here we require $\varepsilon<1$ and $\varepsilon<1/(78K_{9})$, then by sending $N$ to go to infinity and applying the Gronwall's inequality we arrive at,
\begin{eqnarray*}
&&\mathbb{E}\sup_{0\leq s\leq t} \mid \widetilde{Z}^{\varepsilon}(s)\mid^{4} + \mathbb{E} \int_{0}^{T} \mid \widetilde{Z}^{\varepsilon}(s)\mid^{2} \|\widetilde{Z}^{\varepsilon}(s)\|^{2}ds\\
&\leq& M_{2}(\varepsilon, T) \exp\left(c \int_{0}^{t\wedge \tau_{N}} 4c\|u^{0}(s)\|_{L^{4}}^{4} + 8\mid h(s)\mid_{0}^{2}ds\right) =: \widetilde{M}_{2}(\varepsilon, T).
\end{eqnarray*}
Next we use an induction argument to obtain \eqref{(3.10)}. Lemma \ref{Lemm3} confirms the result for $p=1$. Assume that the inequality is true for $p-1$ with its upperbound denoted as $\widetilde{M}_{p-1}(T,\varepsilon)$. Analogous to the previous case, let $\widetilde{\tau}_{N}:= \inf\{t>0: \sup_{0\leq s\leq t} \mid \widetilde{Z}^{\varepsilon}(s)\mid^{2p} + 2p\int_{0}^{T}\mid \widetilde{Z}^{\varepsilon}(s)\mid^{2(p-1)} \|\widetilde{Z}^{\varepsilon}(s)\|^{2}ds>N\}$ then we apply the It$\hat{o}$'s formula first to $\mid \widetilde{Z}^{\varepsilon}(s)\mid^{2}$ and then to the map $x\mapsto x^{p}$ as follows,
\begin{eqnarray*}
d\mid \widetilde{Z}^{\varepsilon}(t)\mid^{2p} &=& p\left(\mid\widetilde{Z}^{\varepsilon}(t)\mid^{2}\right)^{p-1} d\mid\widetilde{Z}^{\varepsilon}(t)\mid^{2} \\
&&+ \frac{1}{2}p(p-1) \mid \widetilde{Z}^{\varepsilon}(t)\mid^{2(p-2)} d<\mid\widetilde{Z}^{\varepsilon}(t)\mid^{2}>_{t}.
\end{eqnarray*}
More precisely,
\begin{flalign*}
&\mid \widetilde{Z}^{\varepsilon}(t \wedge \widetilde{\tau}_{N})\mid^{2p} = -2p \int_{0}^{t\wedge \widetilde{\tau}_{N}} \mid \widetilde{Z}^{\varepsilon}(s)\mid^{2(p-1)} \|\widetilde{Z}^{\varepsilon}(s)\|^{2}ds&\\
&-2p \int_{0}^{t\wedge \widetilde{\tau}_{N}} \mid \widetilde{Z}^{\varepsilon}(s)\mid^{2(p-1)} \left(B(\widetilde{Z}^{\varepsilon}(s), u^{0}(s)), \widetilde{Z}^{\varepsilon}(s)\right)ds&\\
&+ \frac{2p}{\sqrt{2\log \log \frac{1}{\varepsilon}}} \int_{0}^{t\wedge \widetilde{\tau}_{N}} \mid \widetilde{Z}^{\varepsilon}(s)\mid^{2(p-1)} \left(\widetilde{\sigma}(s, \widetilde{Z}^{\varepsilon}(s))dW(s), \widetilde{Z}^{\varepsilon}(s)\right)&\\
&+2p \int_{0}^{t\wedge \widetilde{\tau}_{N}} \mid \widetilde{Z}^{\varepsilon}(s)\mid^{2(p-1)} \left(\widetilde{\sigma}(s,\widetilde{Z}^{\varepsilon}(s))h(s), \widetilde{Z}^{\varepsilon}(s)\right)ds&\\
&+ \frac{p}{2\log \log \frac{1}{\varepsilon}} \int_{0}^{t\wedge \widetilde{\tau}_{N}} \mid \widetilde{Z}^{\varepsilon}(s)\mid^{2(p-1)} \|\widetilde{\sigma}(s, \widetilde{Z}^{\varepsilon}(s)) \|_{L_{Q}}^{2}ds&\\
&+ \frac{2p(p-1)}{\log \log \frac{1}{\varepsilon}} \int_{0}^{t\wedge \widetilde{\tau}_{N}} \left(\widetilde{\sigma}(s, \widetilde{Z}^{\varepsilon}(s)), \widetilde{Z}^{\varepsilon}(s)\right)^{2} \mid \widetilde{Z}^{\varepsilon}(s)\mid^{2(p-2)} ds&\\
&= I_{0}(t\wedge \widetilde{\tau}_{N}) +  I_{1}(t\wedge \widetilde{\tau}_{N}) +  I_{2}(t\wedge \widetilde{\tau}_{N}) +  I_{3}(t\wedge \widetilde{\tau}_{N}) +  I_{4}(t\wedge \widetilde{\tau}_{N}) +  I_{5}(t\wedge \widetilde{\tau}_{N}).&
\end{flalign*}
We take the supremum on time up to $t\wedge \tau_{N}$, and afterwards expectation and estimate,
\begin{eqnarray*}
&& \mathbb{E} I_{1}(t\wedge \widetilde{\tau}_{N}) = 2p \mathbb{E} \int_{0}^{t\wedge \widetilde{\tau}_{N}} \mid \widetilde{Z}^{\varepsilon}(s)\mid^{2(p-1)} b(\widetilde{Z}^{\varepsilon}(s), \widetilde{Z}^{\varepsilon}(s), u^{0}(s))ds\\
&\leq& p \mathbb{E} \int_{0}^{t\wedge \widetilde{\tau}_{N}} \mid\widetilde{Z}^{\varepsilon}(s)\mid^{2(p-1)} \|\widetilde{Z}^{\varepsilon}(s)\|^{2} ds\\
&& + 2p c \mathbb{E} \int_{0}^{t\wedge \widetilde{\tau}_{N}} \sup_{0\leq \ell \leq s} \mid \widetilde{Z}^{\varepsilon}(s)\mid^{2p} \|u^{0}(s)\|_{L^{4}}^{4}ds.
\end{eqnarray*}
With the help of Burkholder-Davis-Gundy inequality,
\begin{eqnarray*}
&&\mathbb{E}\sup_{0\leq s\leq t\wedge \widetilde{\tau}_{N}} I_{2}(s)\\
&\leq& \frac{6p}{\sqrt{2\log \log \frac{1}{\varepsilon}}} \mathbb{E}\left( \int_{0}^{t\wedge \widetilde{\tau}_{N}} \mid \widetilde{Z}^{\varepsilon}(s)\mid^{4(p-1)} \|\widetilde{\sigma}(s, \widetilde{Z}^{\varepsilon}(s))\|_{L_{Q}}^{2} \mid \widetilde{Z}^{\varepsilon}(s)\mid^{2}ds\right)^{1/2}\\
&\leq& \frac{1}{2} \mathbb{E} \sup_{0\leq s\leq t\wedge \widetilde{\tau}_{N}} \mid \widetilde{Z}^{\varepsilon}(s)\mid^{2p} \\
&&+ \frac{9p^{2}}{\log \log \frac{1}{\varepsilon}} \mathbb{E} \int_{0}^{t\wedge \widetilde{\tau}_{N}} \mid \widetilde{Z}^{\varepsilon}(s)\mid^{2(p-1)} \|\widetilde{\sigma}(s,\widetilde{Z}^{\varepsilon}(s))\|_{L_{Q}}^{2}ds\\
&\leq& \frac{1}{2} \mathbb{E}\sup_{0\leq s\leq t\wedge \widetilde{\tau}_{N}} \mid \widetilde{Z}^{\varepsilon}(s)\mid^{2p} + \frac{9p^{2}K_{9}}{\log \log \frac{1}{\varepsilon}} \widetilde{M}_{p-1}(T,\varepsilon)(T+ 2K_{6})\\
&& + 36 p^{2}\varepsilon K_{9} \mathbb{E} \int_{0}^{t \wedge \widetilde{\tau}_{N}} \|\widetilde{Z}^{\varepsilon}(s)\|^{2} \mid \widetilde{Z}^{\varepsilon}(s)\mid^{2(p-1)} ds.
\end{eqnarray*}
Moreover, applying the Cauchy-Schwarz and Young's inequalities lead to,
\begin{flalign*}
&\mathbb{E}I_{3}(t\wedge \widetilde{\tau}_{N}) \leq 2p \mathbb{E}\int_{0}^{t\wedge \widetilde{\tau}_{N}} \mid \widetilde{Z}^{\varepsilon}(s)\mid^{2(p-1)} \|\widetilde{\sigma}(s, \widetilde{Z}^{\varepsilon}(s))\|_{L_{Q}} \hspace{.1cm} \mid h(s)\mid_{0} \hspace{.1cm}\mid \widetilde{Z}^{\varepsilon}(s)\mid ds&\\
&\leq \frac{1}{2} \mathbb{E} \int_{0}^{t\wedge \widetilde{\tau}_{N}} \mid \widetilde{Z}^{\varepsilon}(s)\mid^{2(p-1)} \|\widetilde{\sigma}(s, \widetilde{Z}^{\varepsilon}(s))\|_{L_{Q}}^{2} ds &\\
&+ 2p^{2} \mathbb{E} \int_{0}^{t\wedge \widetilde{\tau}_{N}} \mid h(s)\mid_{0}^{2} \hspace{.1cm} \mid \widetilde{Z}^{\varepsilon}(s)\mid^{2p}ds&\\
&\leq K_{9}\widetilde{M}_{p-1}(T, \varepsilon) \left(\frac{T}{2} + K_{6}\right) + 2\varepsilon K_{9} \log \log \frac{1}{\varepsilon} \mathbb{E} \int_{0}^{t\wedge \widetilde{\tau}_{N}} \mid\widetilde{Z}^{\varepsilon}(s)\mid^{2(p-1)}\hspace{.1cm} \|\widetilde{Z}^{\varepsilon}(s)\|^{2}ds&\\
& + 2p^{2}\mathbb{E} \int_{0}^{t\wedge \widetilde{\tau}_{N}} \sup_{0\leq \ell \leq s} \mid \widetilde{Z}^{\varepsilon}(\ell)\mid^{2p}\hspace{.1cm} \mid h(s)\mid_{0}^{2}ds.&
\end{flalign*}
The same reasoning implies,
\begin{eqnarray*}
\mathbb{E}I_{4}(t\wedge \widetilde{\tau}_{N}) &\leq& \frac{pK_{9}}{\log \log \frac{1}{\varepsilon}} \widetilde{M}_{p-1}(T,\varepsilon) \left(\frac{T}{2} + K_{6}\right)\\
&& + 2p\varepsilon K_{9} \int_{0}^{t\wedge \widetilde{\tau}_{N}}\mathbb{E} \mid \widetilde{Z}^{\varepsilon}(s)\mid^{2(p-1)} \hspace{.1cm} \|\widetilde{Z}^{\varepsilon}(s)\|^{2}ds,
\end{eqnarray*}
and
\begin{eqnarray*}
&&\mathbb{E}I_{5}(t\wedge \widetilde{\tau}_{N}) \leq \frac{2p(p-1)}{\log \log \frac{1}{\varepsilon}} \mathbb{E} \int_{0}^{t\wedge \widetilde{\tau}_{N}} \|\widetilde{\sigma}(s, \widetilde{Z}^{\varepsilon}(s))\|_{L_{Q}}^{2} \mid \widetilde{Z}^{\varepsilon}(s)\mid^{2(p-1)}ds\\
&\leq& \frac{2p(p-1)}{\log \log \frac{1}{\varepsilon}} K_{9} \widetilde{M}_{p-1}(T, \varepsilon) (T+ 2K_{6}) + 8\varepsilon p(p-1) K_{9} \widetilde{M}_{p-1}(T, \varepsilon) \widetilde{M}_{1}(T, \varepsilon).
\end{eqnarray*}
Thus, we have,
\begin{eqnarray*}
&& \frac{1}{2} \mathbb{E} \sup_{0\leq s\leq t\wedge \widetilde{\tau}_{N}} \mid \widetilde{Z}^{\varepsilon}(s)\mid ^{2p} + C(p) \mathbb{E}\int_{0}^{t\wedge \widetilde{\tau}_{N}} \mid \widetilde{Z}^{\varepsilon}(s)\mid^{2(p-1)} \|\widetilde{Z}^{\varepsilon}(s)\|^{2}ds\\
&\leq& M_{p}(T,\varepsilon) + \int_{0}^{t\wedge \widetilde{\tau}_{N}} \mathbb{E}\sup_{0\leq \ell \leq s} \mid \widetilde{Z}^{\varepsilon}(\ell)\mid ^{2p} \left(2pc \|u^{0}(s)\|_{L^{4}}^{4} + 2p^{2} \mid h(s)\mid_{0}^{2} \right)ds,
\end{eqnarray*}
with $C(p)= \left(p-36p^{2}\varepsilon K_{9} -2\varepsilon K_{9} \log \log \frac{1}{\varepsilon}-2p\varepsilon K_{9}\right)$ and
\begin{eqnarray*}
M_{p}(T,\varepsilon) &:=& \frac{p}{\log \log \frac{1}{\varepsilon}} K_{9} \widetilde{M}_{p-1}(T, \varepsilon)\left(\left(11p-\frac{3}{2}\right)T + (22p-3)K_{6}\right)\\
&& + K_{9}\widetilde{M}_{p-1}(T,\varepsilon) \left(\frac{T}{2} + K_{6}\right) + 8\varepsilon p(p-1)K_{9}\widetilde{M}_{p-1}(T,\varepsilon) \widetilde{M}_{1}(T,\varepsilon).
\end{eqnarray*}
Hence by Gronwall's inequality we obtain with bounds, $\varepsilon <1/(K_{9}(36p+2))$ and $\varepsilon<1$,
\begin{eqnarray*}
&& \mathbb{E} \sup_{0\leq s\leq t\wedge \widetilde{\tau}_{N}} \mid \widetilde{Z}^{\varepsilon}(t\wedge \widetilde{\tau}_{N})\mid^{2p} + \mathbb{E} \int_{0}^{t\wedge \widetilde{\tau}_{N}}\mid\widetilde{Z}^{\varepsilon}(s)\mid^{2(p-1)} \|\widetilde{Z}^{\varepsilon}(s)\|^{2}ds\\
&\leq& M_{p}(T, \varepsilon) \exp\left(2pc \|u^{0}(s)\|^{4}_{L^{4}} + 2p^{2} \mid h(s)\mid_{0}^{2}\right) =: \widetilde{M}_{p}(T, \varepsilon).
\end{eqnarray*}
Now letting $N$ to go to infinity, we obtain the result.

\end{document}